\documentclass[a4paper,english,oneside,11pt]{article}
\usepackage{goodstyle}

\title{Morse Theoretic Signal Compression and Reconstruction on Chain Complexes}

\author[1]{Stefania Ebli \thanks{stefania.ebli@epfl.ch}}
\author[1]{Celia Hacker \thanks{celia.hacker@epfl.ch}}
\author[1]{Kelly Maggs \thanks{kelly.maggs@epfl.ch}}
\affil[1]{Laboratory for Topology and Neuroscience\\

École Polytechnique Fédérale de Lausanne (EPFL)}
\date{}
\begin{document}

\maketitle

\begin{abstract}

At the intersection of Topological Data Analysis (TDA) and machine learning, the field of cellular signal processing has advanced rapidly in recent years. In this context, each signal on the cells of a complex is processed using the combinatorial Laplacian, and the resultant Hodge decomposition. Meanwhile, discrete Morse theory has been widely used to speed up computations by reducing the size of complexes while preserving their global topological properties.

In this paper, we provide an approach to signal compression and reconstruction on chain complexes that leverages the tools of algebraic discrete Morse theory. 
The main goal is to reduce and reconstruct a based chain complex together with a set of signals on its cells via deformation retracts, preserving as much as possible the global topological structure of both the complex and the signals.

We first prove that any deformation retract of real degree-wise finite-dimensional based chain complexes is equivalent to a Morse matching. We will  then study how the signal changes under particular types of Morse matching, showing its reconstruction error is trivial on specific components of the Hodge decomposition. Furthermore, we provide an algorithm to compute Morse matchings with minimal reconstruction error.
\end{abstract}

\section{Introduction}

The analysis of signals supported on topological objects such as graphs or simplicial complexes is a fast-growing field combining techniques from topological data analysis, machine learning and signal processing \cite{ortega2018graph,robinson2014topological,barbarossa2018learning}. The emerging field of simplicial and cellular signal processing falls within this paradigm \cite{barbarossa2020topological,schaub2021signal,schaub2021cells}, and here the combinatorial Laplacian $\Delta_n$ plays a pivotal role. In this context, a signal takes the form of a real-valued chain (or cochain) on a chain complex $(\C,\D)$ endowed with a degree-wise inner product. In particular, the eigenvectors of $\Delta_n$, called the Hodge basis, serve as a `topological' Fourier basis to transform a signal into a topologically meaningful coordinate system \cite{ebli2020simplicial,schaub2021signal}. Additionally, the combinatorial Laplacian gives rise to the combinatorial Hodge decomposition~\cite{Eckmann1944}:
$$\C_n = \Ima \D_{n+1} \oplus \Ker \Delta_n \oplus \Ima 
\D_n^\dagger,$$ the components of which each have their own 
topological interpretation
\cite{barbarossa2020topological} and respect the eigendecomposition of $\Delta_n$.

The goal of the paper is to investigate signal compression and reconstruction over cell complexes by combining tools of Hodge theory and discrete Morse theory.
We take an entirely algebraic approach to this problem, working at the level of degree-wise finite-dimensional \textit{based} chain complexes endowed with inner products. 
The classical example is the chain complex of a cell complex equipped with its canonical cellular basis, but more general constructions such as cellular sheaves fit into this framework as well. 
This algebraic perspective not only gives us greater flexibility, but also helps to illuminate connections between Hodge theory and discrete Morse theory that occur only at the level of chain complexes. 

Our approach to compressing and reconstructing signals over complexes involves deformation retracts of based chain complexes, which have the advantage of reducing the size of complexes while preserving their homology. A deformation retract of a chain complex $\C$ onto $\mathbf{D}$ consists of a pair of chain maps $(\Psi, \Phi)$

\begin{center}
    \begin{tikzcd}
        \mathbf{D} \ar[r, shift right, "\Phi"'] & \C \ar[l, shift right, "\Psi"'] \ar[l, loop right, "h"]
    \end{tikzcd}
\end{center}
such that $\Psi \Phi = \id_{\mathbf{D}}$ and a chain homotopy $h : \C \to \mathbf{D}$ between $\Phi\Psi$ and $\id_\C$. In this context, the map $\Psi$ is used to compress the signal $s$ onto the reduced complex $\mathbf{D}$, and $\Phi$ serves to reconstruct it back in $\C$. 
Thus, for every $s\in\C$ one can compute the difference $\Phi \Psi s -s$, called the \emph{topological reconstruction error}, to understand and evaluate how compression and reconstruction changes the signal. 
Among the many topological methods to reduce the size of complexes \cite{robins2011,Singh2007}, discrete Morse theory \cite{Forman1998,Forman2002} provides the perfect tool to efficiently generate such deformation retracts of chain complexes. This technique has already been used with great success in the compression of 3D images \cite{robins2011}, persistent homology \cite{vidit} and cellular sheaves \cite{Curry2016}. In this paper we utilise Sk{\"o}ldberg's \textit{algebraic} version of discrete Morse theory \cite{Sko05, Sko18}. It takes as input a based chain complex $\C$ and, by reducing its based structure with respect to a Morse matching $M$, returns a smaller, chain-equivalent complex $\C^M$.
The first result presented in this article connects the Hodge decomposition of a complex with discrete Morse theory by defining a natural pairing in the Hodge basis. In particular, we show that \textit{any} deformation retract $(\Psi,\Phi,h)$ of degree-wise finite-dimensional, based chain complexes of real inner product spaces can be obtained from a Morse matching over the Hodge basis of a certain sub-complex. This process, called the \textit{Morsification} of $(\Psi,\Phi,h)$, is described in Theorem \ref{morsification}.
In the second part of the paper, we study how the topological reconstruction error associated to a deformation retract $(\Psi,\Phi,h)$ is distributed amongst the three components of the Hodge decomposition. We define a class of deformation retracts $(\Psi,\Phi, h )$, called \emph{$(n,n-1)$-free}, for which the topological reconstruction error has trivial (co)cycle reconstruction. Specifically, they are characterised by the following properties (Theorem \ref{thm:reconstrcution}).
\begin{enumerate}
        \item (Cocycle Reconstruction) A signal $s \in \C_n$ and its reconstruction $\Phi \Psi s$ encode the same cocycle information:
        \begin{center}
            $\mathrm{Proj}_{\Ker \D_{n+1}^\dagger} (\Phi \Psi s - s) = 0$ for all $s \in \C_n$.
        \end{center}
        \item (Cycle Reconstruction) A signal $s \in \C_{n-1}$ and the adjoint of the reconstruction $\Psi^\dagger \Phi^\dagger s$ have the same cycle information:
         \begin{center}
        $\mathrm{Proj}_{\Ker \D_{n-1}} (\Psi^\dagger \Phi^\dagger s - s) = 0$ for all $s \in \C_{n-1}$.
        \end{center}
\end{enumerate}
\noindent Moreover, the Morsification concept defined above simplifies many of the proofs and allows them to be extended into a more general framework (Corollary \ref{cor:weak-reconstruction}).

Finally, we study how the topological reconstruction error of $(n,n-1)$-free deformation retracts can be minimized while maintaining (co)cycle reconstruction. We develop an iterative algorithm to find the retract $(\Psi,\Phi)$ that minimizes the norm of the topological reconstruction error for a given signal $s\in\C$. Our algorithm is inspired by the reduction pair algorithms in  \cite{kaczynski1998homology,vidit,Curry2016} and, like these algorithms, computes a single Morse matching at each step with the additional requirement of minimizing the norm. We show that its computational complexity is linear when the complex is sparse, and discuss bounds on how well the iterative process approximates the optimal deformation retract. Finally, we show computationally that iterating single optimal collapses leads to topological reconstruction loss that is significantly lower than that arising from performing sequences of random collapses.

The paper is structured as follows. In Section \ref{sec:background}, we present the necessary background in algebraic topology, discrete Hodge theory, and algebraic discrete Morse theory, giving the definitions and main results that will be used throughout the paper. 
Section \ref{sec:morsification} introduces the notion of Hodge matching, which allows us to prove that every deformation retract of a degree-wise finite-dimensional based chain complex $\C$ of real inner product spaces is equivalent to a Morse retraction (see Morsification Theorem \ref{morsification}).
In Section \ref{sec:preservation} we investigate the interaction between deformation retracts and Hodge theory. The main results, Theorem \ref{thm:reconstrcution} and Corollary \ref{cor:weak-reconstruction}, utilise the Morsification theorem to prove that $(n,n-1)$-free (sequential) Morse matchings preserve (co)cycles. Section \ref{subsec:sparsification} presents an additional result that explains how the reconstruction $\Phi\Psi s$ can be understood as a sparsification of the signal $s$ (see Lemma \ref{lem:sparsification}).
Finally, Section \ref{sec:algo} is dedicated to presenting algorithms to minimize the topological reconstruction error in case of iterative single pairings (see Algorithms \ref{alg:up} and \ref{alg:k-up}).
\paragraph{Related Work.} Many articles incorporate topology into the loss or reconstruction error function \cite{pmlr-v108-gabrielsson20a,Kim2020,carriere2020,Moor2020}, however, these deal almost exclusive with point cloud data. At the same time discrete Morse theory has been used in conjunction with machine learning in \cite{Xiaoling} for image processing, but not in the context of reconstruction error optimisation. 

The notion of taking duals (over $\Z$) of discrete Morse theoretic constructions is featured in \cite{Forman2002}. There, the dual flow is over $\Z$, whereas we work with adjoint flow over $\R$, for which the orthogonality considerations are somewhat different, as discussed in Appendix \ref{adjoint_retraction_appendix}.

 On the computational side, the articles \cite{vidit, Curry2016, kaczynski1998homology, kaczynski2006computational} involve algorithms to reduce chain complexes over arbitrary PIDs, including those of cellular sheaves but do not investigate the connection with the combinatorial Laplacian (or sheaf Laplacian). Our algorithms are based on the coreduction algorithms of \cite{kaczynski1998homology,kaczynski2006computational}, with the additional requirement of a topological loss minimization.

To the best of our knowledge, the only other contemporary work that examines the link between the combinatorial Hodge decomposition and discrete Morse theory is \cite{contreras2021discrete}, linking the coefficients of the characteristic equation of $\Delta_n$ to the $n$-dimensional paths in an acyclic partial matching. 
\section{Background}\label{sec:background}
In this section, for the sake of completeness, we first recall some basic notions in algebraic topology. We refer the reader to \cite{hatcher} for a more detailed exposition. Then we present the main concepts of algebraic discrete Morse theory and finally, we discuss the foundations of discrete Hodge theory.
\paragraph{Algebraic Discrete Morse Theory.}
For two chain complexes $(\C,\D)$ and $(\mathbf{D},\D')$, a pair of chain maps $\Psi:\C\rightarrow\mathbf{D}$ and $\Phi:\mathbf{D}\rightarrow \C$ are \textit{chain equivalances} if $\Phi\circ \Psi : \C \to \C$ and $\Psi\circ \Phi : \mathbf{D} \to \mathbf{D}$ are  chain homotopic to the identities on $\C$ and $\mathbf{D}$, respectively. Note that this implies that the maps induced on the homology modules by $\Phi$ and $\Psi$ are isomorphisms. The chain equivalences $\Psi$ and $\Phi$ form a \textit{deformation retract} of the chain complexes $\C$ and $\mathbf{D}$ if $\Psi\circ \Phi$ is the identity map on $\mathbf{D}$. 
Deformation retracts will be often depicted as the following diagram.
\begin{center}
    \begin{tikzcd}
        \mathbf{D} \ar[r, shift right, "\Phi"'] & \C \ar[l, shift right, "\Psi"'] \ar[l, loop right, "h"]
    \end{tikzcd}
    \end{center}
With a slight abuse of notation, we denote such deformation retract by the pair $(\Psi,\Phi)$ instead of  $(\Psi,\Phi,h)$.
Throughout the paper we will be working with the following notion of \emph{based} chain complexes, as defined in \cite{Sko18}, which in this context are chain complexes with a graded structure.
\begin{definition}
 Let $R$ be a commutative ring. A \textit{based chain complex} of $R$-modules is a pair $(\C,I)$, where $\C$ is a chain complex of $R$-modules and $I= \{I_n\}_{n\in\N}$ is a set of mutually disjoint sets such that for all $n$ and all $\alpha\in I_n$ there exist $C_\alpha\subseteq \C_n$ such that $\C_n = \bigoplus_{\alpha \in I_n} C_\alpha$.
\end{definition}

Similarly, a based cochain complex is a cochain complex with an indexing set and graded decomposition as above. The components of the boundary operator $\D_n$ are denoted $\D_{\beta, \alpha} : C_\alpha \to C_\beta$ for all $\alpha\in I_n$ and $\beta\in I_{n-1}$. We will refer to the elements of $I_n$ as the \emph{$n$-cells} of $(\C,I)$, and if $\D_{\beta, \alpha} \neq 0$, we say that $\beta$ is a \textit{face} of $\alpha$. If $\C$ is endowed with a degree-wise inner product, we say that $I$ is an orthogonal base if $C_\alpha \perp C_{\beta}$ for all $\alpha \neq \beta \in I$.

\begin{remark}
In this paper, working with combinatorial Hodge theory means that, if not specified otherwise, we restrict our study to degree-wise finite-dimensional chain complexes over $\R$ with an inner product on each of the chain module $\C_n$.\footnote{We leave the original definition here to emphasise that algebraic discrete Morse theory works in more generality.} Moreover, we will refer to degree-wise finite-dimensional based chain complexes as finite-type based chain complexes.
\end{remark}

\noindent The following examples motivate such a choice of terminology for based chain complexes.

\begin{example}\label{ex:dim1-complex}
    In the special case where $(\C,I)$ is a finite-type based chain complex over $\R$ and $C_\alpha \cong \R$ for all $\alpha \in I$, we can think of $I$ as a choice of basis, and each $\D_{\beta,\alpha} \in \mathsf{Hom}(\R,\R) = \R$ as the $(\beta,\alpha)$-entry in the boundary matrix multiplying on the left with respect to such a basis. 
\end{example}

\begin{example}[CW complexes]
The chain complex associated to a finite CW complex with a basis given by its cells is an example of a based chain complex (see \cite{hatcher} for a precise definition of CW complex). For two cells $\sigma, \tau$ in a CW complex $\mc{X}$, denote the degree of the attaching map of $\sigma$ to $\tau$ by $[\sigma:\tau]$ and write $\sigma \triangleright \tau$ whenever they are incident\footnote{Here, incident means that the closure $\overline{\sigma}$ of $\sigma$ contains $\tau$.}. For two incident cells, $\D_{\tau,\sigma}$ is multiplication by $[\sigma:\tau]$.
\end{example}

\begin{example}[Cellular Sheaves] Here we present the main definitions for cellular sheaves, following the more detailed exposition of sheaf Laplacians found in \cite{Hansen2019}. A \emph{cellular sheaf} of finite dimensional Hilbert spaces over a regular\footnote{Regular here indicates that the attaching maps are homeomorphisms.} CW complex $\mc{X}$ consists of an assignment of a vector space $\mc{F}(\sigma)$ to each cell $\sigma \in \mc{X}$ and a linear map $\mc{F}_{\tau \triangleleft \sigma} : \mc{F}(\tau) \to \mc{F}(\sigma)$ to each pair of incident cells $\sigma \triangleright \tau$.
This defines a cochain complex,with $$\C_n = \bigoplus_{\tau \in \mc{X}_n} \mc{F}(\tau),  $$
where $\mc{X}_n$ denotes the set of $n$-cells of $\mc X$, and coboundary maps $\delta_n : \C_n \rightarrow \C_{n+1}$ defined component-wise by $\delta_{\sigma, \tau} = [\sigma:\tau] \mc{F}_{\tau \triangleleft \sigma} : C_\tau \to C_\sigma.$

Using the inner product on $\C_n$ induced by the inner product on each Hilbert space $\mc{F}(\sigma)$, one can define a boundary map $\D_n: \C_{n+1} \rightarrow\C_{n} $ as the adjoint of the coboundary map $\delta_n$. This chain complex is an example of a based chain complex, where the $n$-cells of the base correspond the $n$-cells of the underlying indexing complex.
\end{example}

Discrete Morse theory was originally introduced by Forman in \cite{Forman1998} as a combinatorial version of classical Morse theory.  Here we present its fundamental ideas in a purely algebraic setting, following the exposition in \cite{Sko18}. 

\begin{definition}
 Let $(\C,I)$ be a finite-type based chain complex with base $I$. We denote by $\mc{G}(\C,I)$ the \textit{graph of the complex}, which is the directed graph consisting of vertices $I$ and edges $\alpha \to \beta$ whenever $\D_{\beta,\alpha}$ is non-zero. When clear from the context we will denote $\mc{G}(\C,I)$ by $\mc{G}(\C)$. For a subset of edges $E$ of $\mc{G}(\C)$, denote by $\mc{G}(\C)^E$ the graph $\mc{G}(\C)$ with the edges of $E$ reversed. 
\end{definition}
 
\noindent Using these notions we can define a Morse matching as follows. 

\begin{definition} \label{def:morse-matching}
    An \textit{(algebraic) Morse matching} $M$ on a based complex $(\C,I)$ is a selection of edges $\alpha \to \beta$ in $\mc{G}(\C)$ such that
\begin{enumerate}
    \item each vertex in $\mc{G}(\C)$ is adjacent to at most one edge in $M$;
    \item for each edge $\alpha \to \beta$ in $M$, the map $\D_{\beta, \alpha}$ is an isomorphism;
    \item the relation on each $I_n$ given by $\alpha \succ \beta$ whenever there exists a directed path from $\alpha$ to $\beta$ in $\mc{G}(\C)^M$ is a partial order.
\end{enumerate}
\end{definition}

For context, the third condition corresponds to acyclicity in the classical Morse matching definition, where directed paths akin to gradient flow-lines -- which are non-periodic -- in the smooth Morse theory setting \cite{Milnor}. 

When there is an edge $\alpha\to\beta$ in $M$, we say that $\alpha$ and $\beta$ are \emph{paired} in $M$, and refer to them as a \textit{$(\dim \alpha, \dim \alpha -1)$-pairing}. We use $M^0$ to denote the elements of $I$ that are not paired by $M$, and refer to them as \emph{critical cells} of the pairing. For a directed path $\gamma = \alpha, \sigma_1, \ldots, \sigma_k, \beta$ in the graph $\mc{G}(\C,I)^M$, the \textit{index} $\mathcal{I}(\gamma)$ of $\gamma$ is then defined as
$$ \mathcal{I}(\gamma) = \epsilon_n \D_{\beta, \sigma_n}^{\epsilon_n} \circ \ldots \circ \epsilon_1 \D_{\sigma_2, \sigma_1}^{\epsilon_1} \circ \epsilon_0 \D_{\sigma_1, \alpha}^{\epsilon_0} : C_\alpha \to C_\beta$$ where $\epsilon_i = -1$ if $\sigma_i \to \sigma_{i+1}$ is an element of $M$, and $1$ otherwise. For any $\alpha, \beta \in I$, we define the \textit{summed index} $\Gamma_{\alpha, \beta}$ to be 
$$\Gamma_{\beta,\alpha} = \sum_{\gamma : \alpha \to \beta} \mathcal{I}(\gamma) : C_\alpha \to C_\beta,$$
the sum over all possible paths from $\alpha$ to $\beta$. In the case that there are no paths from $\alpha \to \beta$ then $\Gamma_{\beta,\alpha} = 0$.

The theorem below is the main theorem of algebraic Morse theory. While this theorem was originally proved in \cite{Sko05}, here we state it in the form presented in \cite{Sko18} where it is proved as a corollary of the Homological Perturbation Lemma (\cite{Sko18}, Theorem 1, \cite{Brown65thetwisted,Gugenheim1972OnTC}). This proof provides an explicit description of the chain homotopy $h:\C\rightarrow\C $ that witnesses the fact that the algebraic Morse reduction is a homotopy equivalence.

\begin{theorem}[Sköldberg, \cite{Sko18}] \label{main_lemma}
    Let $(\C,I)$ be a based chain complex indexed by $I$, and $M$ a Morse matching. For every $n \geq 0$ let $$\C^M_n = \bigoplus_{\alpha \in I_n \cap M^0} C_\alpha.$$ The diagram
    \begin{center}
    \begin{tikzcd}
        \C^M \ar[r, shift right, "\Phi"'] & \C \ar[l, shift right, "\Psi"'] \ar[l, loop right, "h"]
    \end{tikzcd}
    \end{center}
    where for $\alpha \in M^0 \cap I_n$ and $x \in C_\alpha$
    \begin{equation*}
    \begin{aligned}
        \D_{\C^M}(x) & = \sum_{\beta \in M^0 \cap I_{n-1}} \Gamma_{\beta,\alpha}(x)     \hspace{3em} \Phi(x) & = \sum_{\beta \in I_n } \Gamma_{\beta,\alpha}(x) 
    \end{aligned}
    \end{equation*}
    and for $\alpha \in I_n$ and $x \in C_\alpha$
    \begin{equation*}
    \begin{aligned}
    \Psi(x) & = \sum_{\beta \in M^0 \cap I_{n-1}} \Gamma_{\beta, \alpha}(x)
    \hspace{3em}
    h(x)  & = \sum_{\beta \in I_{n+1}} \Gamma_{\beta,\alpha}(x)
    \end{aligned}
    \end{equation*}
    is a deformation retract\footnote{In fact the result is stronger. Specifically the maps form a \textit{strong} deformation retract.} of chain complexes.
\end{theorem}

 We refer to the finite-type based chain complex $(\C^M,\D_{C^M}, I \cap M^0)$ as the \textit{Morse chain complex}. Moreover, we call this deformation retract of $\C$ into $\C^M$ the \textit{Morse retraction} induced by $M$.

\begin{example}\label{ex:maps} Given a based chain complex $(\C,I)$ and a single $(n+1,n)$-pairing $M = (\alpha \to \beta)$, Lemma \ref{main_lemma} can be used to get a simple closed form of the updated complex $(\C^M, \D_{\C^M})$ as well as the chain equivalences. We write them explicitly here, and will refer to them throughout the paper.
\begin{itemize}
 \vspace{1em}
    \item For every $\tau, \sigma \in M^0$, the Morse boundary operator is
    $$\D^{\C^M}_{\tau,\sigma} =
    \D_{\tau,\sigma}  - \D_{\tau,\alpha} \D_{\beta,\alpha}^{-1} \D_{\beta,\sigma}.$$
    \item The map $\Psi$ is the identity except at components $C_\alpha$ and $C_\beta$, where it is
    $$ \restr{\Psi_{n}^M}{C_\beta} = \sum_{\tau \in I_{n} \setminus \alpha} - \D_{\tau,\alpha} \D_{\beta,\alpha}^{-1} \hspace{5em} \restr{\Psi_{n+1}^M}{C_\alpha} = 0.
    $$
    \item The map $\Phi$ is the identity except at components $C_\eta$ for each $\eta \in M^0 \cap I_{n+1}$, where it is
    $$\restr{\Phi^M_{n+1}}{\C_\eta}=Id_{\C_\eta} -\D_{\beta,\alpha}^{-1} \D_{\beta,\eta}.$$
\end{itemize}
\noindent Note that these equations are identical to those appearing in \cite{vidit,kaczynski1998homology} in the case that each component $C_\alpha$ is of dimension $1$.
\end{example} 

 When $(\C, I)$ is a finite-type based chain complex of real inner product spaces, the adjoints of the maps in Theorem \ref{main_lemma} play an important role in later sections. Their discrete Morse theoretic interpretation in terms of flow, however, hinges on the orthogonality of the base of $\C$ (see Appendix \ref{adjoint_retraction_appendix}). We will require the following basic result of linear algebra regarding adjoints throughout the paper.
 
 \begin{lemma} \label{adjoint_orthogonal_projection}
    Let $V$ be an finite dimensional inner product space and $W \subseteq V$ be a subspace. The adjoint of the inclusion map $i : W \to V$ is the orthogonal projection $\mathrm{Proj}_W = i^\dag$ onto $W$.
 \end{lemma}

\begin{example}
Let $(\C,I)$ be the canonical based chain complex associated to the cell complex in Figure \ref{fig:ex-pairing}, (left). Following the standard convention of discrete Morse theory, we visually depict a pairing $\alpha \to \beta$ by an arrow running from the cell $\beta$ to the cell $\alpha$. We consider the single $(2,1)$-pairing $M=(\alpha,\beta)$, depicted by the black arrow. Figure \ref{fig:ex-pairing} illustrates how the maps $\Psi^M$ and $\Phi^M$, made explicit Example \ref{ex:maps}, operate on $s\in \C_1$.

\begin{figure}
    \centering
    \includegraphics[width = 0.9\linewidth]{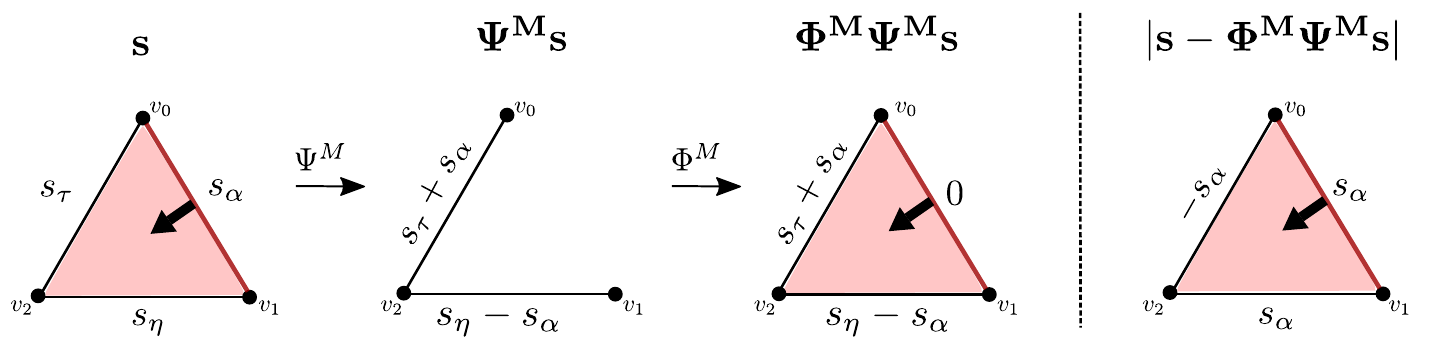}
    \caption{The chain maps $\Psi$ and $\Phi$ operating on a signal $s\in \C_1$.}
    \label{fig:ex-pairing}
\end{figure}

\end{example} 
\begin{remark} 
 Motivated by the emerging field of cellular signal processing, we refer to elements $s \in \C_n$ as \textit{signals} (\cite{barbarossa2020topological, schaub2021signal}). 
 
\end{remark}

In the next definition we introduce the concept of \textit{sequential Morse matching}, an iterative sequence of Morse matchings. This type of matching, unlike a Morse matching, has a low computational cost to reduce the chain complex to a minimal number of critical cells. We discuss this in detail in Section \ref{sec:algo}.

\begin{definition}
A \textit{sequential Morse matching} $\M$ on a based chain complex $(\C,I)$ is a finite sequence of Morse matchings, $M_{(1)},\dots , M_{(n)}$ and bases $I_1, \ldots, I_n$ such that the following conditions hold.
\begin{enumerate}
\item $M_{(1)}$ is a Morse matching on $(\C,I)$.
\item $M_{(j+1)}$ is a Morse matching in $(\C^{M_{(j)}},I_j)$ for every $j\in \{1,\dots,n-1\}$.
\item $\C^{M_{(j)}}$ is a based complex over $I_j \subseteq I_k$ for every $1 
\leq j \leq k \leq n$.
\end{enumerate}
\end{definition}

We denote by $(\C^{\M},\D_{\C_{\M}})$ the based chain complex obtained from $\C$ by iteratively composing the Morse matchings in the sequential Morse matching $\M$, implying that $(\C^{\M},\D_{\C_{\M}}) = (\C^{M_{(n)}}, \D_{\C_{M_{(n)}}})$. Note that in this case, the critical cells of each individual matching in $\M$ form a nested sequence $M_{(1)}^0 \supseteq \cdots \supseteq M_{(n)}^0$. We denote by $\M^0$ the set of \emph{critical cells} of the sequential Morse matching $\M$ and define it to be the set of critical cells in the last Morse matching in the sequence, namely $\M^0 = M_{(n)}^0$.

\paragraph{Combinatorial Laplacians.} For a finite-type based chain complex $\C$ over $\R$ with boundary operator $\D$ and inner products $\langle \cdot,\cdot \rangle_n$ on each $\C_n$, define $\D^\dagger_n : \C_n \to \C_{n+1}$ as the adjoint of $\D_n$, i.e., the map that satisfies $\langle \sigma, \D_n^\dagger \tau \rangle_n = \langle \D_n \sigma,  \tau \rangle_{n-1}$ for all $\sigma \in \C_n$ and $\tau \in \C_{n-1}$. The adjoint maps form a cochain complex
$$  \ldots \xleftarrow{\D_{n+1}^\dagger} \C_n \xleftarrow{\D_n^\dagger} \C_{n-1} \xleftarrow{\D_n^\dagger} \ldots$$
where $(\D^\dagger)^2=0$ follows from the adjoint relation.

\begin{remark}\label{rmk:weights}If $\D_n$ is represented as a matrix in a given basis, and the inner products with respect to that basis are represented as 
$\langle \sigma, \tau \rangle_n = \sigma^T W_n \tau$ where each $W_n$ is a positive-definite symmetric matrix, then the matrix form of the adjoint is given by $\D_n^\dagger = (W_{n}^{-1}) \D_n^T W_{n-1}.$ Note that in our definition the inner product matrix $W_n$ does not necessarily preserve the orthogonality of the standard cellular or simplicial basis in case we are working with cell complexes. In practice, other authors require $W_n$ to be a diagonal matrix to keep the standard basis orthogonal \cite{Horak2013}. In this way the coefficients of $W_n$ can be thought as weights on the $n$-cells, see Appendix \ref{appendix:weights}
\end{remark}

\begin{definition} The \textit{combinatorial Laplacian} is then defined as the sequence of operators $$(\Delta_n = \D_n^\dagger \D_n + \D_{n+1} \D_{n+1}^\dagger: \C_n\longrightarrow\C_n)_{n\geq 0}.$$ For each $n$, the two summands can be further delineated into
\begin{enumerate}
    \item the \textit{$n$-th up-Laplacian} $\Delta_n^+ = \D_{n+1} \D_{n+1}^\dagger :  \C_n \to \C_n$ and
    \item the \textit{$n$-th down-Laplacian} $\Delta_n^- = \D_n^\dagger \D_n : \C_n \to \C_n.$
\end{enumerate}
\end{definition}

\noindent The fundamental results concerning the combinatorial Laplacian were proved by Eckmann in the 1940s \cite{Eckmann1944}.
 
 \begin{theorem}(Eckmann, \cite{Eckmann1944}) If $\C$ is a finite-type based chain complex over $\R$ equipped with an inner product in each degree, then for all $n\geq 0$
    \begin{enumerate}
        \item $H_n(\C) \cong \Ker \Delta_n$, and
        \item $\C_n$ admits an orthogonal decomposition
        \begin{equation}\label{eq:hodge-dec}
        \C_n \cong \Ima \D_{n+1} \oplus \Ker \Delta_n \oplus \Ima \D_n^\dag.
        \end{equation}
 
    \end{enumerate}
 \end{theorem}

The decomposition in the second point, called the \textit{combinatorial Hodge decomposition}, is the finite-dimensional analogue of the Hodge decomposition for smooth differential forms. Two additional orthogonal decompositions associated with adjoints that we will use frequently are
\begin{equation}\label{eq:decomp}
    \C_n = \Ker \D_{n+1}^\dagger \oplus \Ima \D_{n+1} = \Ker \D_n \oplus \Ima \D_n^\dag.
\end{equation}

\paragraph{Singular value decomposition.} Let $V,W$ be real finite-dimensional inner-product spaces. Let $f:V\rightarrow W$ be a linear map and $f^\dagger:W \rightarrow V$ its adjoint. 
The Spectral Theorem states that $f^\dagger f$ and $f f^\dag$ have the same set of real eigenvalues $\Lambda$.  Moreover, the singular value decomposition guarantees that there exist orthonormal bases $\mc{R}(f)$ and $\mc{L}(f)$ of $V$ and $W$ formed by eigenvectors of $f^\dagger f$ and $ff^\dag$ such that for each non-zero $\lambda\in \Lambda$ there exists a unique $v\in\mc{R}(f)$ and a unique $w\in \mc{L}(f)$ such that $$f(v) = \sqrt{\lambda}w.$$ 

We denote by $\mc{L}_+(f)$ and $\mc{R}_+(f)$ the subsets of $\mc{L}(f)$ and $\mc{R}(f)$ respectively corresponding 
to non-zero eigenvalues.
Consider now $f=\D_n:\C_n\to \C_{n-1}, n\geq 0$, the boundary operators associated to a based chain complex. Note that $\mc{L}_+(\D_{n+1})$ and $\mc{R}_+(\D_{n})$, the sets of eigenvectors with positive eigenvalues of $\Delta_n^+=\D_{n+1}\D_{n+1}^{\dagger}$ and $\Delta_n^- = \D_n^\dag\D_n$, form orthonormal bases for $\Ima \D_{n+1}$ and $\Ima \D_n^\dag$, respectively (by Equation (\ref{eq:decomp})). 
In the next section we will see how these eigenvectors together with the Hodge decomposition will allow us to define a canonical Morse matching.

\section{Morsification of Deformation Retracts}\label{sec:morsification}
The aim of this section is to prove that every deformation retract of a finite-type based chain complex $\C$ over $\R$ equipped with degree-wise inner products is equivalent to a Morse retraction, with a canonical choice of basis. We first introduce the notion of the \textit{Hodge matching} on $\C$, a Morse matching defined over the eigenbasis of the combinatorial up and down Laplacians  $\Delta_n^+$ and $\Delta_n^-$. We can see the matching obtained by Hodge decomposition and the eigenvectors of $\Delta_n^+$ and $\Delta_n^-$ as a \emph{canonical} Morse matching. 
\subsection{Hodge Matchings}
The following concept marries the discrete Morse theoretic notion of pairing to the pairing inherent to the eigendecomposition of $\Delta_n^+$ and $\Delta_n^-$, which is intrinsically connected to the Hodge decomposition of a finite real chain complex.

\begin{definition}[Hodge basis]
  Let $\C$ be a finite-type based chain complex over $\R$. A \textit{Hodge basis} of $\C$ is the basis given by $I^{\Delta} = \{I_n^{\Delta}\}_{n\in\N}$, where $$I_n^{\Delta}= \mc{L}_+(\D_{n+1}) \bigcup \mc{R}_+(\D_n) \bigcup \mc{B}(\Ker\Delta_n),$$ for some choice of bases $\mc{L}_+(\D_{n+1}), \mc{R}_+(\D_n)$ and $\mc{B}(\Ker \Delta_n)$. 
\end{definition}
Observe that in the definition above each set in $I_n^{\Delta}$ forms a basis for one of the components in the Hodge decomposition (see Equation \ref{eq:hodge-dec}). Our discussion on the singular value decomposition ensures that Hodge bases always exist.

 \begin{definition}[Hodge matching] 
    Let $\C$ be a finite-type based chain complex of real inner product spaces, and let $I^\Delta$ be a Hodge basis. The \textit{Hodge matching} on $(\C,I^{\Delta})$ is
    $$M^\Delta := \bigcup_i \{ v \in \mc{R}_+(\D_i) \to w \in \mc{L}_+(\D_i) \mid \D_i v = \sigma w, \sigma \neq 0 \}.$$
\end{definition}

\begin{lemma}
    For a finite-type based chain complex $(\C,I^{\Delta})$ of real inner product spaces and $I^\Delta$ be a Hodge basis. The Hodge matching $M^\Delta$ on $(\C,I^{\Delta})$ is a Morse matching and satisfies
    \begin{enumerate}
        \item $(M^\Delta)^0_n = \Ker \Delta_n$, where $\Delta : \C \to \C$ is the combinatorial Laplacian of $\C$ and
        \item $\D^{M^\Delta} = 0.$
    \end{enumerate}
\end{lemma}

\begin{proof}
     The description of orthonormal bases $\mc{L}(\D_n)$ and $\mc{R}(\D_n)$ described at the end Section~\ref{sec:background} implies that each cell is adjacent to at most one other cell in $\mc{G}(\C)^{M^\Delta}$. This means there are no nontrivial paths from any $n$-cell to any other $n$-cell for all $n$ in $\mc{G}(\C)^{M^\Delta}$. Thus, condition (3) in Definition \ref{def:morse-matching} is trivially satisfied, and $M^\Delta$ indeed constitutes a Morse matching. By definition,
    $$ \Ima \D_{n+1} = \Span \mc{L}_+(\D_{n+1}) \, \, \text{and} \, \, \Ima \D_n^\dagger = \Span \mc{R}_+(\D_n),$$
    and all basis elements are paired. The remaining basis elements of $\C_n$ are critical, and constitute $(M^\Delta)^0_n = \Ker \Delta_n$ for all $n$. Since there are no non-trivial paths, $\D^{M^\Delta}$ agrees with the boundary operator $\D$ of $\C$ on $\Ker \Delta$, which is indeed the zero map.

\end{proof}

\noindent We call the data
\begin{center}
    \begin{tikzcd}
        \Ker \Delta \ar[r, shift right, "\Phi^{M^\Delta}"'] & \C \ar[l, shift right, "\Psi^{M^\Delta}"'] \ar[l, loop right, "h"]
    \end{tikzcd}
\end{center}
the \emph{Hodge retraction} of $(\C, I^\Delta)$. Noting that the maps $\Phi^{M^\Delta}$, $\Psi^{M^\Delta}$ are chain equivalences reproves Eckmann's result that $\Ker \Delta$ is isomorphic to the homology $\HH(\C)$ of the original complex.

The same proof also encompasses the case of cellular sheaves discussed in \cite{Hansen2019}. Note that here, a Hodge matching will be over a Hodge base $I^\Delta$ rather than the one specified by the cellular structure of the indexing complex. Nevertheless, since $\Ker \Delta$ does not depend on the choice of base, the result is the same.
\begin{example}
    In Figure \ref{fig:two_optimal_matchings} we depict two different choice of bases -- the standard cellular basis and the Hodge basis -- for the celllular chain complex of the pictured simplicial complex. Two matchings $M$ and $M^\Delta$ are visualized through their corresponding Morse graphs $\mc{G}(\C)^M$ and $\mc{G}(\C)^{M^\Delta}$. The structure of the singular value decomposition of $\D$ and ensuing Hodge matching `straightens out' the connections in the matching graph, as pictured in Figure \ref{fig:two_optimal_matchings}.
\begin{figure}
    \centering
    \includegraphics[width=0.9\linewidth]{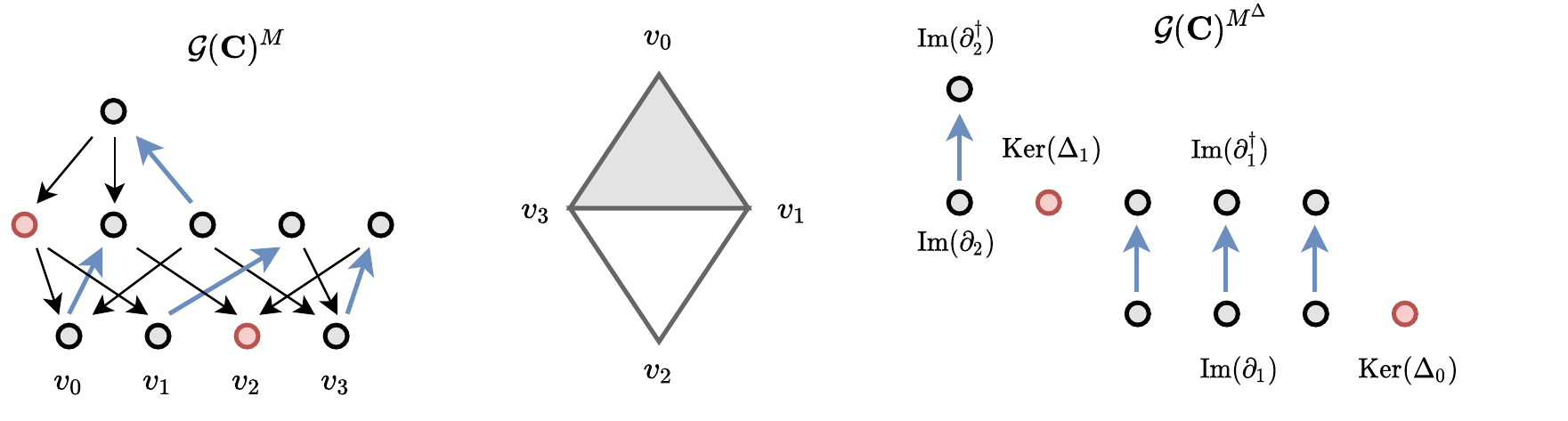}
    \caption{Two choices of bases and Morse matchings for the $\R$-valued chain complex of a simplicial complex. Edges in the Morse matchings are highlighted in blue and critical cells in red. }
    \label{fig:two_optimal_matchings}
\end{figure}

\end{example}

\subsection{Morsification Theorem}

In this section, we say that two deformation retracts 
\begin{center}
    \begin{tikzcd}[column sep = 3em, row sep = 3em]
        \mathbf{D} \ar[r, shift right, "\Phi"'] & \C \ar[l, shift right, "\Psi"'] \ar[l, loop right, "h"]
    \end{tikzcd}
 \hspace{1.5em}and 
  \hspace{1.5em}
    \begin{tikzcd}[column sep = 3em, row sep = 3em]
        \mathbf{D}' \ar[r, shift right, "\Phi'"'] & \C' \ar[l, shift right, "\Psi' "'] \ar[l, loop right, "h' "]
    \end{tikzcd}
\end{center}
are \emph{equivalent} if there exist isomorphisms of chain complexes, $f:\mathbf{D}\rightarrow\mathbf{D}'$ and $g:\C\rightarrow\C'$ such that the diagrams
    \begin{center}
    \begin{tikzcd}[column sep = 3em, row sep = 3em]
        \mathbf{D} \ar[d, "f"', "\cong"]  & \C \ar[d, "g", "\cong"'] \ar[l, "\Psi"'] \\
        \mathbf{D}'  & \C' \ar[l, "\Psi'"] 
    \end{tikzcd}
    \hspace{3em}
    \begin{tikzcd}[column sep = 3em, row sep = 3em]
        \mathbf{D} \ar[d, "f"', "\cong"] \ar[r, "\Phi"] & \C \ar[d, "g", "\cong"']  \\
        \mathbf{D}' \ar[r, "\Phi'"'] & \C' 
    \end{tikzcd}
    \end{center}
commute. Our goal is to show that any deformation retraction of finite-type chain complexes of real inner product spaces is isomorphic to a Morse retraction (Theorem \ref{morsification}). 

In the special case that $\C = \C'$ and $g$ is the identity, the commutativity of the diagrams above implies that
\begin{equation} \label{reconstruction_equality}
    \Phi'\Psi' = \Phi f^{-1} f \Psi = \Phi\Psi.
\end{equation}
Thus, to study the topological reconstruction error of a deformation retract, it is enough to study that of an equivalent deformation retract of the original complex. Two equivalent deformation retracts over a shared domain $\C$ may have different homotopies, however, they are related by
$$ \D h + h \D= 1 - \Phi \Psi  = 1 - \Phi' \Psi' = \D h' + h' \D.$$

The main theorem of this section relies on the observation that deformation retracts share a number of characteristics with projection maps in linear algebra i.e. a linear endomorphism $P : V \to V$ of a vector space $V$ satisfying $P^2 = P$. For any projection map, there exists a decomposition $V = \Ima P \oplus \Ker P$ such that $P$ can be decomposed as
 $$P = 1_{\Ima P} + 0 : \Ima P \oplus \Ker P \to \Ima P \oplus \Ker P.$$ The following lemma describes an analogous structure for real chain complexes, where a deformation retraction plays the role of a projection.

 \begin{lemma}\label{projection_lemma_chain_complexes} For any deformation retract 
    \begin{center}
    \begin{tikzcd}
        \mathbf{D} \ar[r, shift right, "\Phi"'] & \C \ar[l, shift right, "\Psi"'] \ar[l, loop right, "h"]
    \end{tikzcd}
    \end{center}
    of chain complexes over $\R$, 
    \begin{equation} \label{deformation_splitting}
        \C = \Ker\Psi \oplus \Ima\Phi.
    \end{equation}
    as chain complexes. 
\end{lemma}

\begin{proof}
    The deformation retraction condition $\Psi\Phi = \id_\textbf{D}$ implies that $$(\Phi_n \Psi_n)^2 =\Phi_n \Psi_n \Phi_n \Psi_n = \Phi_n \Psi_n,$$ 
   i.e., each component $\Phi_n \Psi_n$ of $\Phi \Psi$ is a projection operator. Thus there is a splitting of vector spaces
    $$\C_n = \Ker(\Phi\Psi)_n \oplus \Ima(\Phi\Psi)_n$$
    for each $n$. Since $\Phi\Psi$ is a chain map, the decomposition above commutes with the boundary operator of $\C$, whence $$\C = \Ker\Phi\Psi \oplus \Ima\Phi\Psi$$ as chain complexes. Lastly, $\Psi$ is surjective and $\Phi$ is injective since $\Psi \Phi = \id_\textbf{D}$, implying that $\Ima\Phi \Psi = \Ima\Phi$ and $\Ker\Phi \Psi = \Ker\Psi$.
\end{proof}

The decomposition defined in Equation \ref{deformation_splitting} has an interesting interpretation when passing to homology: all of the non-trivial homology of $\C$ arises from the $\Ima\Phi$ component of the decomposition. 
One way to think of this decomposition is that $\Ker\Psi$ is the component of $\C$ that is discarded by the deformation retraction, whereas $\Ima \Phi$ is preserved.

\begin{lemma}
    Under the hypotheses of Lemma \ref{projection_lemma_chain_complexes}
    \begin{enumerate} 
        \item $\HH(\C) \cong \HH(\Ima\Phi)$, and
        \item $\HH(\Ker\Psi) = 0$.
    \end{enumerate}
\end{lemma}

\begin{proof}
    Since $\Psi$ is a weak equivalence, $\HH(\C) \cong \HH(\mathbf{D})$. Since $\Psi\Phi=\id_{\mathbf{D}}$, $\Phi$ is injective, so $\mathbf{D} \xrightarrow{\Phi} \Ima\Phi$ is an isomorphism of chain complexes, proving point $(1)$. Since $\C = \Ker\Psi \oplus \Ima\Phi$ by Equation \ref{deformation_splitting}, it follows that $\HH(\Ker\Psi)=0$.
\end{proof}

\begin{theorem}[Morsification] \label{morsification} 
    Any deformation retract
    \begin{center}
    \begin{tikzcd}
        \mathbf{D} \ar[r, shift right, "\Phi"'] & \C \ar[l, shift right, "\Psi"'] \ar[l, loop right, "h"]
    \end{tikzcd}
    \end{center}
    of finite-type chain complexes of real inner product spaces is equivalent to a Morse retraction $(\Psi^\mc{M},\Phi^\mc{M})$ over $\C$ .
\end{theorem}

\begin{notation}
    We refer to the pairing $\mc{M}$ in this theorem as the \textit{Morsification} of a deformation retract.
\end{notation}

\begin{proof}
    Define a pairing $\mc{M} = \widetilde{M}^\Delta \sqcup \widehat{M}$ on $\C$ as the union of a Hodge pairing $\widetilde{M}^\Delta$ on $\Ker\Psi$ (which is given the subspace inner product) and the trivial pairing $\widehat{M}$ on $\Ima \Phi$. We previously showed that $\C = \Ker\Psi \oplus \Ima\Phi$ and $\HH(\C) = \HH(\Ima\Phi)$, implying that $\HH(\Ker\Psi)=0$. Consequently, all the basis elements in $\Ker\Psi$ are paired by the Hodge pairing, and further, the Morse retraction maps 
    \begin{center}
        \begin{tikzcd}
        \HH(\Ker\Psi) \cong 0 \ar[r, shift right, "\Phi^{\widetilde{M}^\Delta}"'] & \Ker\Psi \ar[l, shift right, "\Psi^{\widetilde{M}^\Delta}"']
    \end{tikzcd}
    \end{center}
    defined by the matching $\widetilde{M}^\Delta$ are trivial.
    
    On the other hand, since $\widehat{M}$ is the trivial pairing, the entirety of $\Ima\Phi$ is critical in the pairing $\mc{M}$. Further, the Morse boundary operator $\D^{\widehat{M}}$ is the same as the boundary operator on $\C$, implying $\C^M = \Ima \Phi$ and that the maps
    \begin{center}
        \begin{tikzcd}
        \C^M \cong \Ima\Phi \ar[r, shift right, "\Phi^{\widehat{M}}"'] & \Ima\Phi \ar[l, shift right, "\Psi^{\widehat{M}}"']
    \end{tikzcd}
    \end{center}
    are identities. We conclude that $\Phi^\mc{M} \Psi^\mc{M} = i_{\Ima \Phi} \circ \pi_{\Ima \Phi}$, where $i_{\Ima \Phi} : \Ima_\Phi \hookrightarrow \C$ is the inclusion. 
    
    Now we show that this is equivalent to the original deformation retract. To do so, first note that $\Phi : D \to \Ima \Phi$ is an isomorphism. We then need to show that the following diagram
    \begin{center}
        \begin{tikzcd}[column sep=2em, row sep=3em]
            & \C \ar[dl, "\Psi"'] \ar[dr, "\Psi^{\mc{M}}"] & \\
            \mathbf{D} \ar[rr, "\Phi"', "\cong"] & & \Ima\Phi
        \end{tikzcd}
    \end{center}
    commutes. For any $(s, \Phi(t)) \in \C = \Ker\Psi \oplus \Ima\Phi$, we have
    $$\Phi \Psi(s, \Phi(t)) = (\Phi \Psi(s), \Phi \Psi \Phi(t)) = (0,\Phi(t)) = i \circ \pi_{\Ima \Phi}(s, \Phi(t)) = \Phi^{\mc{M}} \Psi^{\mc{M}}(s,\Phi(t))$$ as required. Finally, to see that
     \begin{center}
        \begin{tikzcd}[column sep=2em, row sep=3em]
            & \C   & \\
            \mathbf{D} \ar[ur, "\Phi"] \ar[rr, "\Phi"', "\cong"] & & \Ima\Phi \ar[ul, "\Phi^{\mc{M}}"']
        \end{tikzcd}
    \end{center}
    commutes simply note that $\Phi^\mc{M}$ is the inclusion map.
\end{proof}

\begin{remark}
    When the original deformation retract comes from a Morse matching, the subspace $\Ima \Phi = \Ima \Phi \Psi = \Ker(1-\Phi\Psi)$ is the space of \textit{flow-invariant chains} used by Forman in his foundational articles \cite{Forman1998,Forman2002}. The difference here is that these chains are linear combinations of genuine critical cells, albeit for a Morse matching in a new base.
\end{remark}

It is not difficult to see that the Morsification of a deformation retract is unique up to a choice of bases in the eigenspaces of $\Delta^+$ and $\Delta^-$, and that each such choice produces equivalent deformation retracts. Combining Theorem \ref{morsification} with Equation \ref{reconstruction_equality}, we get a simple expression for the reconstruction error of a deformation retract in terms of the paired cells in its Morsification.

\begin{corollary} \label{reconstruction_closedform}
    Any deformation retract
    \begin{center}
    \begin{tikzcd}
        \mathbf{D} \ar[r, shift right, "\Phi"'] & \C \ar[l, shift right, "\Psi"'] \ar[l, loop right, "h"]
    \end{tikzcd}
    \end{center} 
    of finite-type chain complexes of real inner product spaces and Morsification $\mc{M}$
    $$1 - \Psi\Phi = \sum_{ \alpha \in I^\mc{M} \setminus \mc{M}^0} i_\alpha \circ \pi_\alpha$$
\end{corollary}

\begin{proof}
    By Equation \ref{reconstruction_equality} and Theorem \ref{morsification}, we have
    $$ 1 - \Phi \Psi =  1-  i_{\Ima \Phi} \circ \pi_{\Ima \Phi} = i_{\Ker \Phi \Psi} \circ \pi_{\Ker \Phi \Psi} = \sum{ \alpha \in I^\mc{M} \setminus \mc{M}^0} i_\alpha \circ \pi_\alpha$$
    which proves the statement, noting that the paired cells in $\mc{M}$ span $\Ker \Psi$.
\end{proof}

In the case that the deformation retract arises from a Morse matching on a based complex, the Morsification construction will most likely alter the base. However, the number of pairings and critical cells in each dimension are related, as described in the following proposition. 

\begin{notation}
 For a sequential Morse matching $\M$ on a based chain complex $(\C, I)$, let $\M_n^-$ and $\M_n^+$ denote the elements of $I_n$ that are the union of all start and endpoints respectively of edges in each of the matchings $M_{(i),n}$ for all $i$. This means that $$I_n = \M_n^- \sqcup \M_n^0 \sqcup \M_n^+.$$ Further, let $$ \lvert \M_n^* \rvert = \sum_{\alpha \in \M_n^*} \dim C_\alpha$$ where $* \in \{ +,-,0 \}$, and the subscript $n$ refers to the dimension of the cells.   
\end{notation}

\begin{proposition} \label{prop:weakmorse_dim}
 Let $\M $ be a  sequential Morse matching on a finite-type based chain complex $(\C,I)$ of real inner product spaces and $\mc{M}$ be its Morsification. Then
 $$\lvert \mc{M}_n^* \rvert = \lvert \M_n^* \rvert $$
 for $* \in \{ +,-,0 \}$,
in each dimension $n \geq 0$.
\end{proposition}

\begin{proof}
    By Theorem \ref{morsification} we know that $\C^{\M} \cong \C^{\mc{M}}$, implying that the dimensions spanned by critical cells $$\abs{\M_n^0} = \dim \C_n^{\M} = \dim \C_n^{\mc{M}} = \abs{\mc{M}_n^0}$$
    are equal for all $n$. This implies that
    \begin{equation} \label{up_down_summation}
        \abs{\M_n^+} + \abs{\M_n^-} = \dim \C_n - \dim \C_n^{\M}= \abs{\mc{M}_n^+} + \abs{\mc{M}_n^-}
    \end{equation}
    where we have used the identity $\dim \C_n = \abs{\M_n^+} + \abs{\M_n^-} + \abs{\M_n^0}$. 
    
    Since the chain complex is concentrated in non-negative degrees, cells in dimension 0 can be paired only with elements in dimension $1$, implying that $\abs{\M_0^-} = \abs{\mc{M}_0^-} = 0$. Combining this with Equation \ref{up_down_summation} we conclude that $\abs{\M_0^+} = \abs{\mc{M}_0^+}$. The bijection between cells paired up in dimension $i$ with those paired down in dimension $i+1$ then implies that $$\abs{\M_1^-} = \abs{\M_0^+} = \abs{\mc{M}_0^+} = \abs{\mc{M}_1^-},$$ and, again using Equation \ref{up_down_summation}, that $\abs{\M_1^+} = \abs{\M_1^+}$. By inductively performing this procedure, we prove the result for all $n$ as required.
\end{proof}

It is not difficult to see that two equivalent Morse retractions of $\C$ must have the same Morsification. Thus the above proposition then implies that when two sequential Morse retractions $\M$ and $\M'$ of a complex $\C$ under two different bases $I$ and $I'$ are equivalent, there are equalities between the number of dimensions paired up $\abs{\M_n^+} = \abs{\M_{n}^{'+}}$ and down $\abs{\M_n^-} = \abs{\M_n^-}$ for all $n$. Notably, this occurs independently of the bases $I$ and $I'$.

\section{(Co)cycle Preservation and Sparsification}\label{sec:preservation}
Discrete Morse theory aims to reduce the dimension of a chain complex while preserving its homology. Meanwhile, for combinatorial Hodge theory, understanding the effect of deformation on the components of the Hodge decomposition is of equal importance. However, because of the `adjointness' inherent in the Hodge decomposition, neither chain or cochain maps between two complexes usually respect the grading of the Hodge decomposition. 

Here, we define a different notion of preservation by examining the effect of applying either $\Phi \Psi$ or $\Psi^\dagger \Phi^\dag$ to an element $s \in \C_n$. For a pair of chain maps
 \begin{center}
    \begin{tikzcd}
        \mathbf{D} \ar[r, shift right, "\Phi"'] & \C \ar[l, shift right, "\Psi"'] 
    \end{tikzcd}
    \end{center}
we define the \textit{topological reconstruction error} at $s\in \C$ as $\Phi\Psi s - s \in \C$. The goal of this section is to examine the projection of $\Phi\Psi s - s$ on the different components of the Hodge decomposition. In particular, we describe which components of the signal are preserved and discarded by $\Phi \Psi$ when the deformation retract arises from a $(n,n-1)$-free Morse matching, a special type of (sequential) Morse matchings described in the next section. 
Further, we show that for such matchings the reconstruction $\Phi\Psi s$ (or $\Psi^\dag\Phi^\dagger s$) is supported only on the critical cells, and serves to sparsify the data on the original complex while preserving the (co)cycle information. 

\subsection{\texorpdfstring{$\boldsymbol{(n, n-1)}$}{(n,n-1)}-free Matchings}

\begin{definition}   A Morse matching $M$ is said to be $\mathit{(n,n-1)}$\textit{-free} if $\lvert M_n^- \rvert = 0$. 
\end{definition}

  An equivalent condition is that $\lvert M_{n-1}^+ \rvert = 0$. Put simply, a Morse matching is $(n,n-1)$-free if no $n$-cells are paired with $(n-1)$-cells. In what follows, the mantra is that preservation of (co)cycle information in dimension $n-1$ (or $n$) is equivalent to absence of such pairings. We define an $\mathit{(n,n-1)}$\textit{-free sequential Morse matching} $\M=(M_{(1)},\dots,M_{(k)})$ to be a sequential Morse matching where all $M_i$ are $(n,n-1)$-free Morse matchings.

\begin{figure}
    \centering
    \includegraphics[width = 0.5\linewidth]{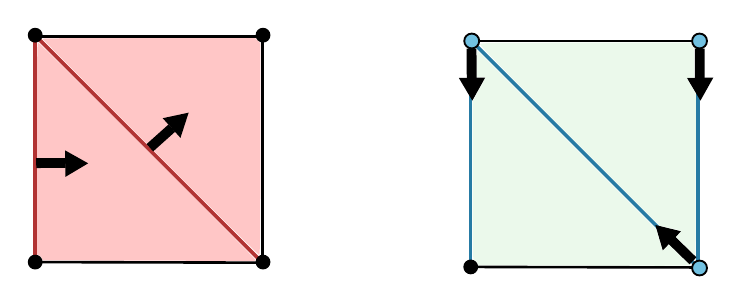}
    \caption{Two Morse matchings -- the left is $(1,0)$-free and the right is $(2,1)$-free.}
    \label{fig:updown_example}
\end{figure}

\begin{example}
    Figure \ref{fig:updown_example} shows a $(1,0)$-free and a $(2,1)$-free matching. The matchings are computed on the cellular chain complex of the depicted cell complex, based with the standard cellular basis. We visually depicted the pairings in the macthings by black arrows. Note that being $(n,n-1)$-free does not necessarily prohibit all $n$ or $(n-1)$-cells from appearing in the matching, implying that $(n,n-1)$-free matchings can still lead to dimension reduction of both $\C_n$ and $\C_{n-1}$.
\end{example}

\begin{example}
    If $\C$ is finite-type chain complex of real inner product spaces such that $\D_n = 0$, then the Hodge matching $M^\Delta$ is $(n,n-1)$-free for some choice of Hodge basis $I^\Delta$.
\end{example}

The corollary below, which follows immediately from Proposition \ref{prop:weakmorse_dim}, shows that the property of being $(n,n-1)$-free is not an artifact of our choice of basis. Namely, if two Morse matchings are equivalent, then either they are both $(n,n-1)$-free or neither is.

\begin{corollary}\label{cor:free-morsification}
    A sequential Morse matching $\M$ on a based chain complex $(\C,I)$ is $(n,n-1)$-free if and only if  its Morsification $\mathcal{M}$ is $(n,n-1)$-free. 
\end{corollary}

\subsection{(Co)cycle Preservation for \texorpdfstring{$\boldsymbol{(n,n-1)}$}{(n,n-1)}-free Matchings}
\label{subsec:preservation}
The following reconstruction theorem shows that both the topological reconstruction error of the deformation retract and its adjoint are supported on non-kernel components of the Hodge decomposition.
\begin{theorem}[Reconstruction]\label{thm:reconstrcution} 
    Suppose that $M$ is a Morse matching on a finite-type based chain complex $(\C,I)$ of real inner product spaces. Let 
        \begin{center}
    \begin{tikzcd}
        \C^M \ar[r, shift right, "\Phi"'] & \C \ar[l, shift right, "\Psi"'] \ar[l, loop right, "h"]
    \end{tikzcd}
    \end{center}
    be the deformation retract given by Theorem \ref{main_lemma}. Then 
    \begin{enumerate}
        \item for all $s \in \C_n$, $$\mathrm{Proj}_{\Ker \D_{n+1}^\dagger} (\Phi \Psi s - s) = 0 \text{, and}$$
        \item for all $s \in \C_{n-1}$, $$\mathrm{Proj}_{\Ker \D_{n-1}} (\Psi^\dagger \Phi^\dagger s - s) = 0$$
    \end{enumerate}
if and only if $M$ is a $(n,n-1)$-free matching.
\end{theorem}

\begin{proof}
We first prove that if $M$ is a $(n,n-1)$-free matching, then conditions $(1)$ and $(2)$ hold.
    If $M_n^- = \emptyset$, then there are no paths in $\mc{G}(\C)^M$ from an $(n-1)$-cell to an $n$-cell. Theorem \ref{main_lemma} then implies that $h_{n-1}(x) = 0$ for all $\alpha \in I_{n-1}$ and $x \in C_\alpha$, whence
    \begin{equation} \label{eq1}
    (\Phi \Psi - 1)_n = \D_{n+1} h_n + h_{n-1} \D_n = \D_{n+1} h_n.
    \end{equation}
    The first claim now follows from the orthogonal decomposition $$\C_n = \Ker \D_{n+1}^\dagger \oplus \Ima \D_{n+1}.$$ 
    
    The argument above also shows that $h^\dag_{n-1} = 0$, since the adjoint of the zero map is the zero map. By taking the adjoint of Equation  \ref{eq1} one dimension lower, it then follows that
    \begin{equation*}
        (\Psi^\dagger \Phi^\dagger - 1)_{n-1} = (\Phi \Psi-1)^\dag_{n-1} =\D_{n-1}^\dagger  h_{n-2}^\dagger + h_{n-1}^\dagger \D_{n}^\dagger = \D_{n-1}^\dagger  h_{n-2}^\dag.
    \end{equation*}
    The second claim is then a consequence of the orthogonal decomposition $\C_{n-1} = \Ker \D_{n-1} \oplus \Ima \D_{n-1}^\dag.$

For the other direction we will prove the contrapositive statement. It is sufficient to show that if the Morse matching is not $(n,n-1)$-free, then there exists $s\in \C_n$ such that $$\mathrm{Proj}_{\Ker \D_{n+1}^\dagger} (\Phi \Psi s - s) \neq  0.$$ The Morse matching $M$ is $(n,n-1)$-free if and only if its Morsification $\mathcal{M}$ is $(n,n-1)$-free (Corollary \ref{cor:free-morsification}) and, further,  $1-\Phi^M\Psi^M=1-\Phi^{\mathcal{M}}\Psi^{\mathcal{M}}$ (Equation \ref{reconstruction_equality}). Therefore, it is sufficient to prove the contrapositive statement for the Morsification.

Since the Morsification is not $(n,n-1)$-free, there exists an $(n,n-1)$-pair $\alpha \to \beta$ such that $\D_{\beta,\alpha}$ is an isomorphism. Recall that by \ref{reconstruction_closedform}, we have that $(\Phi^{\mathcal{M}}\Psi^{\mathcal{M}}-1)x=x$ for $x \in C_\alpha$. The orthogonal decomposition of $\C_n$ implies that $$x=\mathrm{Proj}_{\Ker \D_{n}}x + \mathrm{Proj}_{\Ima\D_n^\dagger}x.$$ Applying $\D_n$ and using the fact that $\D_n(x) \neq 0$, we obtain

$$0 \neq \D_n \mathrm{Proj}_{\Ker \D_{n}}x + \D_n \mathrm{Proj}_{\Ima_n^\dagger}x=\D_n \mathrm{Proj}_{\Ima \D_n^\dagger}x.  $$ Since $\Ima\D_n^\dagger \subseteq \Ker \D_{n+1}^\dag$, this implies that 
$$ 0\neq \mathrm{Proj}_{\Ker \D_{n+1}^\dagger}x=\mathrm{Proj}_{\Ker \D_{n+1}^\dagger} (\Phi^{\mathcal{M}} \Psi^{\mathcal{M}} -1)x=\mathrm{Proj}_{\Ker \D_{n+1}^\dagger} (\Phi^{M} \Psi^{M} -1)x,$$
which proves our statement.

\end{proof}

The utility of the theorem above is that an $(n,n-1)$-free matching $M$ reduces the dimension of $\C_n$, while perfectly preserving the $n$-cocycles of a signal $s \in \C_n$ under the reconstruction $\Phi_n \Psi_n$. The extent of this reduction depends on the $(n+1,n)$-pairs in $M$. Indeed, the direct sum of the components $\bigoplus_{\alpha \in M_n^+} C_\alpha$ of $n$-cells in such pairs is isomorphic to the subspace $\Ker \Psi_n$ discarded by the deformation retract. One way to see this is using the fact that the Morsification has the same pair structure as the sequential Morse matching, and the Morsification $\Phi^{\mc{M}}$ is zero on non-critical cells.

 If, on the other hand, one is interested in preserving the cycle information of a signal $s \in \C_{n-1}$, then one can use the adjoint maps $\Phi^\dagger \Psi^\dag$ to perform a similar procedure. Namely, an $(n,n-1)$-free matching $M$ will perfectly preserve the $(n-1)$-cycle part of $s$ under the reconstruction $\Psi_{n-1}^\dagger \Phi_{n-1}^\dag$. Analogously to the dual case, the extent of reduction depends on the $(n-1,n-2)$-pairings, where the subspace $\bigoplus_{\alpha \in M_{n-1}^-} C_\alpha$ is isomorphic to the discarded subspace $\Ker \Phi_{n-1}^\dag$.

Using Morsification, we can extend the (co)cycle reconstruction theorem to $(n,n-1)$-free sequential Morse matchings.

\begin{corollary}\label{cor:weak-reconstruction}
Let $\M$ be a sequential Morse matching on a based chain complex $(\C,I)$. Then the (co)cycle preservation conditions $(1)$ and $(2)$ of Theorem \ref{thm:reconstrcution} hold if and only if $\M$ is $(n,n-1)$-free.
\begin{proof}
By Corollary \ref{cor:free-morsification} we know that $\M$ is $(n,n-1)$-free if and only if its Morsification $\mathcal{M}$ is $(n,n-1)$-free. Further, we know that $$1-\Phi^{\M}\Psi^{\M}=1-\Phi^{\mathcal{M}}\Psi^{\mathcal{M}}$$ by Equation \ref{reconstruction_equality}. Then the statement follows by applying Theorem~\ref{thm:reconstrcution} to $\C$ and $\mathcal{M}$. 
\end{proof}

\end{corollary}

One may wonder whether there is a proof by induction that follows directly from Theorem \ref{thm:reconstrcution}. The problem with using induction is that each chain complex in the sequential Morse matching has a different Hodge decomposition, and that the maps between them do not necessarily respect the grading. So Theorem \ref{thm:reconstrcution} implies the (co)cycle preservation conditions will be satisfied between the $i$-th and $(i+1)$-th chain complexes but not necessarily between $\C$ and $\C^{\mc{M}}$.

In the general case of deformation retracts that do not arise from a Morse matching, combining Theorem~\ref{thm:reconstrcution} and Corollary \ref{cor:free-morsification} yields the following.

\begin{corollary}
Let $(\Phi,\Psi)$ be a deformation retract of based finite-type chain complexes $(\C,I)$ and $(\mathbf{D},I')$ of real inner product spaces. Then the (co)cycle preservation conditions $(1)$ and $(2)$ of Theorem \ref{thm:reconstrcution} hold if and only if the Morsification $\mathcal{M}$ associated to $(\Phi,\Psi)$ is $(n,n-1)$-free.
\end{corollary}

\subsection{Sparsification for \texorpdfstring{$\boldsymbol{(n,n-1)}$}{(n,n-1)}-free Matchings}\label{subsec:sparsification}

In the previous section, we showed how a signal's projection onto each Hodge component is related to that of its reconstruction. In addition, one would like to know how the reconstructed signal sits in the complex with respect to the base on which the Morse matching is constructed. 

In this section we will show that, for a $(n,n-1)$-free (sequential) Morse matching, the image of $\Phi_n \Psi_n$ is supported only on the critical cells $M_n^0$ of $I_n$. Intuitively, applying $\Phi_n \Psi_n$ can be thought of as a form of sparsification which preserves one of either cycles or cocycles (Theorem \ref{thm:reconstrcution}).

\begin{lemma} \label{inclusion_lemma}
    Let $M$ be an $(n,n-1)$-free matching of an orthogonally based finite-type chain complex $(\C,I)$ of real inner product spaces. Then
    \begin{enumerate}
        \item $\Phi_n : \C_n^M \to \C$ and
        \item $\Psi_{n-1}^\dagger : \C_{n-1}^M \to \C$
    \end{enumerate}
    are subspace inclusions and, thus, isometries.
\end{lemma}

\begin{proof}
    By Theorem \ref{main_lemma}
    $$ \Phi_n = \sum_{\alpha \in M_n^0} \sum_{\beta \in I_n} \Gamma_{\beta,\alpha}.$$ A path in $\mc{G}(\C)^M$ starting at an $n$-dimensional critical cell must first step down a dimension. Since $M$ is $(n,n-1)$-free, it cannot return to dimension $n$. This shows that the only paths starting at critical cells in dimension $n$ are trivial and hence
    $$ \Phi_n(x) = \sum_{\beta \in I_n} \Gamma_{\beta,\alpha}(x) = x $$
    for all $x \in \C_\alpha$, $\alpha \in M_n^0$.
    
    For point (2), recall that
    $$ \Psi_{n-1} = \sum_{\alpha \in M_{n-1}^0} \sum_{\beta \in I_{n-1}} \Gamma_{\alpha,\beta}.$$ When $\alpha \in \M_{n-1}^0$, all non-trivial paths in $\mc{G}(\C)^M$ from $\beta \in I_{n-1}$ to $\alpha$ must pass through dimension $n$. However, this is impossible since $M$ is $(n,n-1)$-free, implying all paths out of critical cells in dimension $(n-1)$ to cells in dimension $(n-1)$ are trivial and $\sum_{\beta \in I_{n-1}} \Gamma_{\alpha,\beta} = \pi_{\alpha}$. This yields
    $$ \Psi_{n-1} = \sum_{\alpha \in M_{n-1}^0} \pi_\alpha =  \pi_{\C^M}.$$ 
    
    According to Lemma \ref{adjoint_orthogonal_projection}, the inclusion $i : \C^M \to \C$ is the adjoint of the \textit{orthogonal} projection $\mathrm{Proj}_{\C^M}$, and is not necessarily the same as the categorical projection $\pi_{\C^M}$. However, the condition that the base $I$ is orthogonal, implies that $\C^M$ is is indeed orthogonal to $\C/\C^M$, and that $\Psi_{n-1}^\dag$ is the inclusion map $i : \C^M \hookrightarrow \C$ as required.
\end{proof}

\begin{remark}
    The condition that the base is orthogonal is also important for having a discrete Morse theoretic interpretation of the adjoint in terms of backwards flow within the Morse graph $\mc{G}(\C)^M$. We explain this perspective in detail in  Appendix \ref{adjoint_retraction_appendix}.
\end{remark}

Given that the composition of a sequence of inclusions of sub-spaces is again an inclusion, Lemma \ref{inclusion_lemma} holds equally well for $(n,n-1)$-free \textit{sequential} Morse matchings. 

\begin{corollary}[Sparsification]\label{lem:sparsification}
    Let $\M$ be an $(n,n-1)$-free sequential Morse matching of an orthogonally based chain complex $(\C,I)$. Then
    \begin{enumerate}
        \item $$\Phi_n^{\M} \Psi_n^{\M}(s) \in \bigoplus_{\alpha \in M^0 \cap I_n} C_\alpha \text{ for all } s \in \C_n$$
        \item $$\Psi^{\M\dagger}_{n-1} \Phi^{\M\dagger}_{n-1}(s) = \bigoplus_{\beta \in M^0 \cap I_{n-1}} C_\beta \text{ for all } s \in \C_{n-1}.$$
    \end{enumerate}
\end{corollary}

\begin{proof}
    By definition we know that $$\Psi_n^{\M}(s) \in \bigoplus_{\alpha \in \M^0 \cap I_n} C_\alpha = \C_n^{\M} \hspace{1em} \text{and} \hspace{1em} \Phi^{\M\dag_{n-1}}(s) \in  \bigoplus_{\beta \in \M^0 \cap I_{n-1}} C_\beta = \C_{n-1}^{\M}.$$ The result then follows from Lemma \ref{inclusion_lemma}, which implies that both $\Phi_{n}^{\M}$ and $\Psi_{n-1}^{\M\dagger}$ are compositions of subspace inclusions.
\end{proof}

\begin{example} In this example we consider the based chain complex $\C$ associated to the cell complex $\mc{X}$ in Figure~\ref{fig:cartoon-up-collapses}-A. We work with the standard basis generated by the $n$-cells and the standard boundary operator $\D_*$. The signal $s\in \C_1$ is obtained by randomly sampling from $[0,1]$. We consider the $(1,0)$-free matching $M$ in Figure~\ref{fig:cartoon-up-collapses}-C, where there are two $1$-cells are paired with two $2$-cells, denoted by the arrows. All the other cells are critical. 

In Figure~\ref{fig:cartoon-up-collapses}-A we show how the signal $s$ is transformed by the maps $\Phi^M$ and $\Psi^M$ induced by the $(1,0)$-free matching $M$. The absolute value of the reconstruction error, $| s-\Phi^M\Psi^M |$ is shown in Figure~\ref{fig:optimal-up-collapses}-B. As proved in Theorem~\ref{thm:reconstrcution}, we observe in  Figure~\ref{fig:optimal-up-collapses}-D that the reconstructed signal $\Phi^M\Psi^Ms$ is perfectly preserved on $\Ker\D_1=\Ker\Delta_1 \oplus \Ima \D_1^{\dagger}$, and all changes in the reconstructed signal are contained in $\Ima\D_2$. Note that $\Phi_1^M \Psi_1^Ms$ is supported only on the critical $1$-cells as proved in Lemmas \ref{lem:sparsification} and \ref{inclusion_lemma}.

\begin{figure}[!ht]
\vspace{2pt}
\begin{center}
\includegraphics[width=0.90\textwidth]{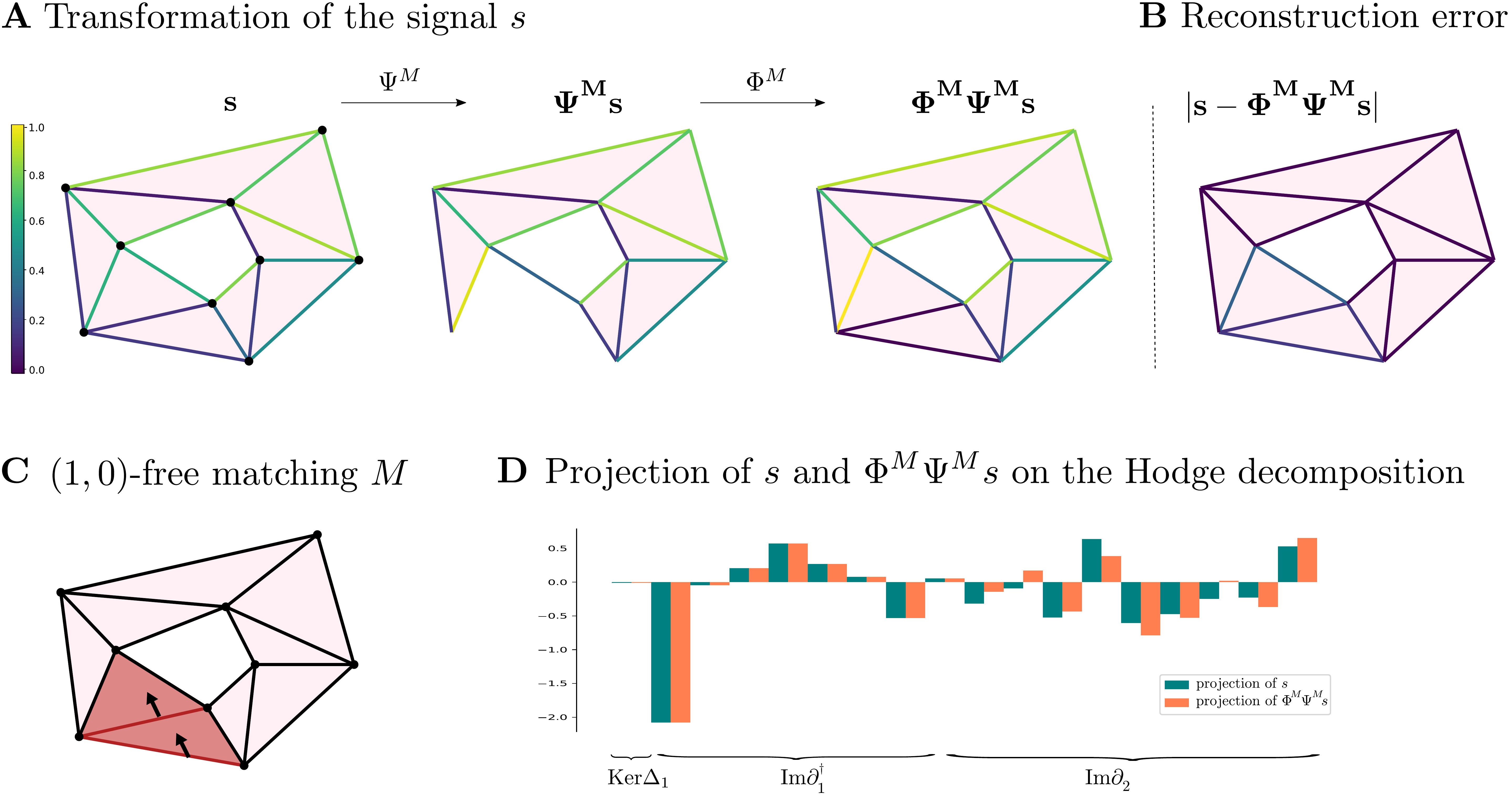}
  \caption{The life-cycle and reconstruction error of a signal $s \in \C$ in the standard basis of a simplicial complex under the maps associated to a Morse matching.}

\label{fig:cartoon-up-collapses}
\end{center}
\vspace{-.25in}
\end{figure}
\end{example}

\section{Algorithms and Experiments}\label{sec:algo}
The goal of this section is to reduce a based complex $(\C,I)$ together with a signal $s \in \C$ (or set of signals $\mc{S} \subset \C$) via a sequential Morse matching while trying to minimize the norm of the topological reconstruction error. 

We propose the following procedure to iterativly reduce a based chain complex $(\C,I)$ with signal $s$ via a sequential Morse matching. The method is inspired by the classical reduction pair algorithm described in \cite{kaczynski2006computational,kaczynski1998homology} but differs in the optimization step in (1).
\begin{enumerate}
    \item If $\D\neq0$, select a single pairing $\alpha \to \beta$ in $(\C,\D)$ minimizing $\norm{ s - \Phi \Psi s}$.
    \item Reduce $\C$ to $\C^M$ and repeat with $\C = \C^M$ and $\D=\D^{C^M}$.
\end{enumerate}
Note that this procedure differs as well from that of Nanda et al.\ which, in the context of both persistent homology \cite{vidit} and cellular sheaves \cite{Curry2016}, requires an actual Morse matching.
The details of the algorithm are provided in Section \ref{sec:optimal_pairig} (see Algorithm~\ref{alg:up} and Algorithm~\ref{alg:k-up}), where we also show that their computational complexity is linear in the number of $(n+1)$-cells. 
In Section \ref{sec:loss} we discuss the behaviour of the norm of the topological reconstruction error when performing this type of iterated reduction. In Section \ref{sec:convergence} we prove that such an algorithm converges to a based chain complex with the minimal number of critical cells. Finally, in Section \ref{sec:experiments} we provide experiments on synthetic data.
\begin{remark}
Since in most of the applications $\dim C_\alpha = 1$ for all $\alpha \in I$, we will work with this assumption throughout the following sections. Thus, without loss of generality, we will refer to the elements of $I_n$ as a basis of $\C_n$ and denote $\D_{\beta,\alpha}=[\alpha:\beta]$ (see Example \ref{ex:dim1-complex} for more details).
\end{remark}

\subsection{Algorithms for Optimal (sequential) Morse Matchings}\label{sec:optimal_pairig}

For a pair of chain maps
\begin{center}
    \begin{tikzcd}
        \mathbf{D} \ar[r, shift right, "\Phi"'] & \C \ar[l, shift right, "\Psi"']
    \end{tikzcd}
    \end{center}
between based chain complex with inner product on each $\C_n$ and $\mathbf{D}_n$, and a signal $s \in \C_n$, define the \textit{topological loss} of the maps $(\Phi, \Psi)$ over $s$ to be the norm of the topological reconstruction error
\begin{equation}\label{eq:topo_loss}
\mc{L}_s(\Psi,\Phi) = \langle s - \Phi \Psi s, s - \Phi \Psi s \rangle^{1/2}_{\C_n}=\norm{ s - \Phi \Psi s}_{\C_n}.
\end{equation}
For a finite subset $\mc{S} \subset \C_n$, the loss is defined to be the sum
$$\mc{L}_\mc{S}(\Psi,\Phi) = \sum_{s \in \mc{S} } \mc{L}_s(\Psi,\Phi)$$
of the individual losses. The loss of a single collapse can be given a closed form by using Theorem \ref{main_lemma}, in the case of a deformation retract associated to a Morse matching.

Specifically, suppose we have a single $(n+1,n)$-pairing $\alpha \to \beta$. Theorem \ref{main_lemma} implies that the homotopy $h$ maps $\beta$ to $-\frac{1}{[\alpha:\beta]}\alpha$ and is zero elsewhere. For a signal $s \in \C_{n}$, using the equations developed in Example \ref{ex:maps}, we have
\begin{equation}\label{eq:topolo_loss_compact} 
    \mc{L}_s(\Psi, \Phi) = \lVert (1 - \Phi\Psi)s \rVert_{\C_n} = \big\lVert  \D_n h_{n}  s \rVert_{\C_n} = \Bigg\lVert \dfrac{s_\beta }{[\alpha : \beta]}\cdot \D_{n+1}(\alpha)  \Bigg\rVert_{\C_n}
\end{equation}
where $s_\beta$ is the component of $s$ on basis element $\beta$. Similarly, for a signal $s \in \C_{n+1}$ we have a \emph{dual} topological loss
\begin{equation}
\mc{L}_s(\Phi^\dag, \Psi^\dag) = \lVert (1 - \Psi^\dagger \Phi^\dag)s \rVert_{n+1} = \lVert \D_{n+1}^\dagger h_{n}^\dagger  s \lVert_{n}
\end{equation}
If $I$ is an orthogonal basis for $\C$, Theorem \ref{adjoint_main_lemma} implies that we can write this loss as
\begin{equation*}
\mc{L}_s(\Phi^\dag, \Psi^\dag) = \Bigg\lVert s_\alpha \dfrac{\D_{n+1}^\dag(\beta)}{[\alpha : \beta] } \Bigg\rVert_{\C_{n+1}}
\end{equation*}

Note that to write a compact form for Equation (\ref{eq:topo_loss}), in case $M$ is not a single Morse matching, one needs to sum over all possible non-trivial paths in Theorem \ref{main_lemma}. Therefore finding the matching $M$ minimizing this norm would be computationally expensive, if not infeasible. On the other hand, it is not hard to find the single $(n+1,n)$-pairing $\alpha \to \beta$ minimizing the topological loss in Equation (\ref{eq:topolo_loss_compact}). Therefore, as a first approach towards finding an approximate solution of the problem, we begin by studying optimal matchings by restricting to iterated single pairings. 
\begin{remark}\label{rmk:dual-loss}
Naturally, one can ask the same questions about finding the optimal pairing minimizing the topological loss for $\Psi^\dagger\Phi^\dagger s -s$. Given the duality of the problem, we will present algorithms and experiments only for $\Phi\Psi s -s$. The algorithms and computations for the dual topological loss can be found by dualizing the chain and boundary maps.
\end{remark}
 Given a finite-type based chain complex $(\C,I)$ of real inner product spaces and a signal $s$ on the $n$-cells, our goal is now to find the the $(n+1,n)$-pairing $\alpha \to \beta$ minimizing the topological loss in Equation (\ref{eq:topolo_loss_compact}).
Computing the minimum and its arguments for a single pair boils down to storing for each $(n+1)$-cell $\tau$ in the basis the face $\sigma$ where the quantity $$\frac{|s_{\sigma}|}{|[\tau:\sigma]|}\norm{\D_{n+1}\tau}_n$$ is minimal, and choosing among all the ($n+1$)-cells the one realizing the minimum of $\mc{L}_s$.

\begin{example}\label{rmk:min-signal}
Consider the based chain complex associated to a simplicial complex $\mc{X}$ with basis induced by its cells and $\D_*$ the standard boundary operator. Let $s$ be a signal on the $n$-cells. The minimum of the reconstruction loss $\mc{L}_s$ in Equation (\ref{eq:topolo_loss_compact}) is then realized on the $n$-cell $\beta$, where $|s_{\beta}|$ is minimum, paired with any of its cofaces $\alpha$. Note that the minimum and its argument might not be unique.

\end{example}

Following the idea above, Algorithm~\ref{alg:up} returns a single $(n+1,n)$-pairing $\alpha \to \beta$ that minimizes the topological loss for a given based chain complex $(\C,I)$ and signal $s$.

\begin{algorithm}
\footnotesize{
\caption{Perform a single optimal pairing}\label{alg:up}
\textbf{Input} {A based chain complex $\C$ with basis $I$, a signal on $\C_n$, $\D_{n+1}$, the non-zero $n+1$-boundary.}
\textbf{Output} {A a single $(n+1,n)$-pairing $\alpha \to \beta$ which minimize the topological loss.}
\begin{algorithmic}[1]

\Function{OptimalPairing}{$\C$, $I$, \var{signal}, $\D_{n+1}$}
\vspace{10pt}
\For{each $n+1$-cell $\tau$ in $I_{n+1}$} 
\State {\var{OptCol}[$\tau$]=0}
\Comment{\var{OptCol} keeps track of the face which realizes the optimal collapse on $\tau$}
\EndFor
\vspace{10pt}
\For{each $n+1$-cell $\tau$ in $I_{n+1}$} 
\State {\var{ValOptCol}[$\tau$]=$\infty$} \hspace{0.1pt} 
\Comment{\var{ValOptCol} keeps track of the value of the optimal collapse on $\tau$}
\EndFor
\vspace{10pt}  
\For{each $n+1$-cell $\tau$ in $I_{n+1}$} 
\For{each face $\xi$ of $\tau$ in $\mc{F}_{\tau}$} 
\State  $x \gets \dfrac{|\text{\var{signal}}[\xi]| }{|[\tau: \xi]| }\norm{\D_{n+1} (\tau) }_{\C_n}$ 
\State \var{ValOptCol}[$\tau$]$\gets$minimum({$x,\text{\var{ValOptCol}}[\tau]$})
\EndFor
\State $\sigma\gets$ 
random$\left(\argmin\limits_{\xi \in \mc{F}_{\tau}}\left(\dfrac{|\text{\var{signal}}[\xi]| }{|[\tau: \xi]| }\norm{\D_{n+1} (\tau) }_{\C_n}=\text{\var{ValOptCol}}[\tau]\right)\right)$

\Comment{$\sigma$ is randomly chosen among the faces of $\tau$ which have minimal reconstruction loss}
\State \var{OptCol}[$\tau$]$\gets \sigma$
\EndFor
\State \var{TotalMin}$\gets$minimum({\var{ValOptCol}})
\Comment The value \var{TotalMin} is the minimum reconstruction loss.
\State $\alpha \gets$ random({$\argmin\limits_{\tau \in I_{n+1}}(\text{\var{ValOptCol}=\var{TotalMin}})$})
\Comment{The $n+1$ cell $\alpha$ to collapse is randomly chosen among the $n+1$ cells where the reconstruction loss is minimal.}
\State $D \gets $ \var{OptCol}[$\alpha$]
\Comment{The $n$ cell $\beta$ to collapse is the face of $\tau$ obtaining minimal reconstruction loss.}
\\ \
\Return{$(\alpha,\beta)$}
\EndFunction

\end{algorithmic}
}
\end{algorithm}
The computational complexity of Algorithm~\ref{alg:up} is $O(pc^2)+O(p)$, where $p=\dim \C_{n+1}$ and $c=\max_{\tau\in I_{n+1}} |\D_{n+1}\tau|$. The first term follows from the fact that we need to iterate through all the $(n+1)$-cells and their faces, computing the minimum of lists of size at most $c$. The second summand follows from the fact that the final step of the algorithm requires computations of the minimum of a list of size at most $c$. Since the first summand dominates the second one, the computational complexity of Algorithm~\ref{alg:up} is $O(pc^2)$. We assume that in most of the computations we are dealing with sparse based chain complexes, i.\ e.\ based chain complexes in which the number of $n$-cells in the boundary of an $(n+1)$-cell is at most a constant $ c \ll p$. In this case the computational complexity of Algorithm~\ref{alg:up} is $O(p)$.

In practice, one would like to further reduce the size of a based chain complex. In Algorithm~\ref{alg:k-up} we provide a way to perform a sequence of single optimal collapses.  For a based chain complex $\C$ and a signal $s$, the algorithm computes at each iteration a single optimal pairing $(\alpha,\beta)$ and it updates $(\C,\D)$ to $(\C^M, \D_{C_M})$ and the signal $s$ to $\Psi^M s$. 
\begin{algorithm}
\footnotesize{
\caption{Perform $k$ single optimal pairings}\label{alg:k-up}
\textbf{Input} {A based chain complex $\C$ with basis $I$, a signal on $\C_n$, $\D_{n+1}$ the non-zero $(n+1)$-boundary and parameter $k$ of the number of single optimal collapses to perform.}
\ \\
\textbf{Output}{A based chain complex $\C^M$ with basis $I^M \subseteq I$ and its boundary $\D_{C_M}$ obtained by iteratively computing $k$ optimal pairings starting from $\C$.}
\begin{algorithmic}[1]
 \Function{k-OptimalPairings}{$\C$, $I$, \var{signal}, $\D_{n+1}$, k}
\State $i \gets 1$

\While{$ i \leq k$}
\State $(\alpha, \beta) \gets $\Call{OptimalUpCollapse}{$\C$, $I$, \var{signal}, $\D_{n+1}$}
\State $(\C, \D,I) \gets (\C^M, \D_{C_M},I_M)$
\State \var{signal}$ \gets \Psi( $\var{signal})
\State $i \gets i + 1$
\EndWhile
\\ \
\Return $\C$ , $\D$
\EndFunction
\end{algorithmic}}
\end{algorithm}

In fact, Algorithm~\ref{alg:k-up} consists of the classical reduction pair algorithm proposed in \cite{kaczynski2006computational,kaczynski1998homology} with the additional step of the loss minimization. If applied only to a $(n,n-1)$-free sequential Morse matching, Algorithm \ref{alg:k-up} will converge to a based chain complex with given dimensions, as we prove in Proposition \ref{free_convergence}. Otherwise, if applied to cells of every size, it allows us to reduce a chain complex up to a minimal number of critical $n$-cells, as proved in \cite{kaczynski1998homology}. We state again this result in Section \ref{sec:convergence}. At the same time, the algorithm constructs a $(n,n-1)$-free sequential Morse matching, therefore the original signal is perfectly reconstructed on part of the Hodge decomposition, as proved in Theorem~\ref{thm:reconstrcution}. Finally, a further justification for the choice of this iterative algorithm, is that the loss on the original complex is bounded by the sum of the losses in the iterative step. We further discuss this in the next section.

\subsection{Conditional Loss}\label{sec:loss}
The computational advantages outlined above are dictated by the fact that Algorithm \ref{alg:k-up} iteratively searches for optimal pairings. One important detail to understand is then how the loss function interacts with such iterated reductions. For a diagram of chain maps

\begin{center}
    \begin{tikzcd}
        \mathbf{E} \ar[r, shift right, "\Phi'"'] & \mathbf{D}  \ar[l, shift right, "\Psi'"'] \ar[r, shift right, "\Phi"'] & \C \ar[l, shift right, "\Psi"']
    \end{tikzcd}
    \end{center}
and $s \in \C_n$, define the conditional loss to be
$$\mc{L}_s(\Psi', \Phi' \mid \Psi, \Phi) = \mc{L}_{\Psi(s)}(\Psi', \Phi') = \lVert \Psi s - \Phi' \Phi' \Psi s \rVert_{\mathbf{D}_n}. $$ In practice, we will generate a sequential Morse matching by taking a series of collapses and optimising the conditional loss at each step.

\begin{lemma}
    Let $C,D$, and $E$ be inner product spaces and suppose we have a diagram of linear maps
    \begin{center}
        \begin{tikzcd}
            E \ar[r, shift right, "\phi'"'] & D \ar[l, shift right, "\psi'"'] \ar[r, shift right, "\phi"'] & C \ar[l, "\psi"', shift right]
        \end{tikzcd}
    \end{center}
    where $\phi$ is an isometry. Then for all $s \in C$ we have
    \begin{equation*}
        \lVert (1-\phi \phi' \psi' \psi)s \rVert_C \leq \lVert (1-\phi \psi)s \rVert_C + \lVert (1-\phi' \psi') \psi(s) \rVert_D.
    \end{equation*}
\end{lemma}

\begin{proof}
    Using the triangle inequality and the fact that $\phi$ is an isometry, we have
    \begin{align*}
        \lVert (1-\phi \phi' \psi' \psi) s \rVert_C & = \lVert (1-\phi \psi)s + \phi (1 - \phi' \psi')\psi(s) \rVert_C \\
        & \leq \lVert (1-\phi \psi)s \rVert_C + \lVert \phi (1 - \phi' \psi')\psi(s) \rVert_C \\
        & = \lVert (1-\phi \psi)s \rVert_C + \lVert (1 - \phi' \psi')\psi(s) \rVert_D
    \end{align*}
    as required.
\end{proof}

The following corollary justifies the approach of minimizing the conditional loss at each step. It states that the loss on the original complex will be bounded by the sum of the conditional losses. Note that the same result and proof also work for the adjoint case where $s \in \C_{n-1}$, as long as the complex is orthogonally based.

\begin{corollary}
    Suppose we have a diagram of chain maps
    \begin{center}
    \begin{tikzcd}
        \mathbf{E} \ar[r, shift right, "\Phi'"'] & \mathbf{D}  \ar[l, shift right, "\Psi'"'] \ar[r, shift right, "\Phi"'] & \C \ar[l, shift right, "\Psi"']
    \end{tikzcd}
    \end{center}
    where each step arises from an $(n,n-1)$-free Morse matching. Then for all $s \in \C_n$
    $$ \mc{L}_s(\Psi'\Psi,\Phi\Phi') \leq \mc{L}_s(\Psi,\Psi) + \mc{L}_s(\Psi',\Phi' \mid \Psi,\Phi)$$
\end{corollary}

\begin{proof}
    In the Sparsification Lemma \ref{lem:sparsification}, we showed that taking $(n,n-1)$-matchings implied that $\Phi_n, \Phi'_n$ are isometries. The result then follows from applying the lemma above. 
\end{proof}
\subsection{Reduction Pairings and Convergence}\label{sec:convergence}

The following proposition ensures that the reduction pair algorithm proposed in \cite{kaczynski1998homology}, which is the foundation of Algorithm \ref{alg:k-up}, converges in a finite (and pre-determined) number of steps to the homology of $\C$. This advantage of being able to maximally reduce a based complex is in contrast with the well-studied NP-hard problem \cite{Joswig2006ComputingOM} of finding Morse matchings. In this section, we will prove an analogous result for $(n,n-1)$-free matchings.

\begin{theorem}[Kaczynski et al.\ \cite{kaczynski1998homology}]\label{prop:convergence}
    Let $(\C,I)$ be a finite-type based chain complex over $\R$, where $\dim C_\alpha = 1$ for all $\alpha \in I$. The iteration of the following procedure
    \begin{enumerate}
    \item If $\D\neq0$, select a single pairing $\alpha \to \beta$ in $(\C,\D)$.
    \item Reduce $\C$ to $\C^M$ and repeat with $\C = \C^M$ and $\D=\D_{C_M}$.
\end{enumerate}
    converges to the complex $\HH(\C)$ with $\D = 0$ after
    $$N = \dfrac{1}{2} \sum_n (\dim \C_n  - \dim \HH_n(\C))$$
    steps.
\end{theorem}

  To prove a similar result for $(n,n-1)$-free matchings, we first prove two lemmas describing how the dimensions of the summands in the Hodge decomposition of $\C^M$ relate to those of $\C$ when $M$ is a single pairing. 

\begin{lemma} \label{image_prop}
    Let $M = (\alpha \to \beta)$ be an $(n+1,n)$-pairing of a based complex $(\C,I)$. Then
    $$\Ima \D_{n}^M = \Ima \D_{n}$$
\end{lemma}

\begin{proof}
    Since no $(n-1)$-cells are deleted by $M$, $\C_{n-1} = \C_{n-1}^M$. The formulas in the background section in Example \ref{ex:maps} show that $
    \D_n^M = \restr{\D_n}{\C_n^M}$, implying that $\Ima \D^M_n + \D_n(C_\beta) = \Ima \D_n$. To prove the statement it then suffices to show that  $\D_n(C_\beta)$ is contained in $\Ima \D_n^M$. Using $\D_n \D_{n+1}=0$ and the fact that $\D_{\alpha,\beta}$ is an isomorphism, we then have that
    \begin{align*}
        0 = \D_n(\D_{n+1}(C_\alpha)) & = \D_n( \D_{\beta,\alpha} (C_\alpha) + \sum_{\tau \in I_n \setminus \beta} \D_{\tau,\alpha} (C_\alpha))\\
        \Rightarrow \D_n( C_\beta) & = - \D_n( \sum_{\tau \in I_n \setminus \beta} \D_{\tau,\alpha} (C_\alpha)) \subseteq \Ima \D_n^M.
    \end{align*}
    which proves the result.
\end{proof}

Note that while the images of both $\D_n^M$ and $\D_n$ agree, the eigendecomposition of their correspondent up- and down-Laplacians may not be related in a straightforward way. In other words, the combinatorial Laplacian eigenbases for $\C_{n-1}$ and $\C_{n-1}^M$ can be rather different, even though the corresponding summands of their Hodge decompositions have the same dimensions.

\begin{lemma} \label{eigen_collapse}
    Let $M = (\alpha \to \beta)$ be an $(n+1,n)$-pairing of a finite-type based complex $(\C,I)$ of real inner product spaces. Then
    \begin{equation} 
    \dim \Ima (\D_i^M)^\dagger = \dim \Ima \D^M_{i}  = \begin{cases} \dim \Ima \D_i - \dim C_\beta & i = n+1 \\
    \dim \Ima \D_i & \text{else}
    \end{cases}
    \end{equation}
\end{lemma}

\begin{proof}
    The left equality is a basic property of adjoints. For the right equality, note that (1) $\C \simeq \C^M$ implies $\dim \Ker \Delta^M_i = \dim \Ker \Delta_i$ for all $i$ and (2)Lemma \ref{image_prop} implies that $\dim \Ima (\D_n^M)^\dagger = \dim \Ima \D_n^\dag$. Together these imply that $$ \dim \C_n - \dim \C_n^M = \dim \Ima \D_{n+1} - \dim \Ima \D_{n+1}^M = \dim C_\beta.$$ Equivalently, this says that $\dim \Ima \D_{n+1}^\dagger - \dim \Ima (\D_{n+1}^M)^\dagger = \dim C_\alpha$, and now all of the change in dimension from $\C$ to $\C^M$ has been accounted for.
\end{proof}

We can now state the convergence theorem for the $(n,n-1)$-sequential Morse matchings over $\R$ in Algorithm \ref{alg:k-up}. Along with homology, $\dim \Ima\D_n$ and $\dim\Ima\D^{\dagger}_n$ provide a (strict) upper bound on how many pairings we can make in an $(n,n-1)$-free sequential Morse matching. 

\begin{proposition}[Convergence] \label{free_convergence}
 Let $(\C, I)$ be a finite-type based chain complex over $\R$ with inner products. Then Algorithm \ref{alg:k-up} for $(n,n-1)$-free Morse matchings converges to a based chain complex $\mathbf{D}$ such that 
 \begin{equation*}
     \mathbf{D}_i \cong \begin{cases} \HH(\C_i) \oplus \Ima \D_i^\dagger & i = n\\
     \HH(\C_i) \oplus \Ima \D_{i+1} & i = n-1\\
     \HH(\C_i) & \text{else}
     \end{cases}
 \end{equation*}
 where $\D_i^\mathbf{D} = 0$ for all $i \neq n$.
\end{proposition}

\begin{proof}
Given the conditions on the basis assumed at the beginning of the section, $\D_{\alpha,\beta}$ is an isomorphism if and only if it is a multiplication by a non-zero element of $\R$. Hence, $\D_i = 0$ if and only if we are not able to make any more $(i,i-1)$-pairings, implying the process must converge to some complex $\mathbf{D}$ with $\D^\mathbf{D}_i =0$ for all $i \neq n$. Since $\mathbf{D}$ is weakly equivalent to $\C$, this proves that $\mathbf{D}_i = \HH_i(\mathbf{D}) = \HH_i(\C)$ for all $i \not\in \{ n,n-1 \}$.

By Lemma \ref{eigen_collapse}, each $(n+1,n)$-pairing reduces the dimension of $\Ima \D_{n+1}$ by $1$, and each $(n-1,n-2)$-pairing reduces the dimension of $\Ima \D_{n-1}^\dag$ by $1$. One can iterate the process of either $(n+1,n)$-pairing or $(n-1,n-2)$-pairing, until $\dim \Ima \D_{n+1}= 0$ or $\dim \Ima \D_{n-1}^\dagger =0$ respectively. Thus, the isomorphism in the Lemma follows from this itarative process and from the Hodge decomposition of $\mathbf{D}_i$.
\end{proof}

\subsection{Experiments}\label{sec:experiments}
In this section we provide examples of how algorithms \ref{alg:up} and \ref{alg:k-up} can be applied to compress and reconstruct signals on synthetic complexes. Moreover, we show computationally that the topological reconstruction loss of a sequence of optimal pairings given by algorithm \ref{alg:k-up} is significantly lower than the loss when performing sequences of random collapses (see Figure~\ref{fig:optimal-up-collapses} and Figure~\ref{fig:topo_error}). Our main goal is to provide a proof of concept for the theoretical results and algorithms of this article rather than an exhaustive selection of experiments. The code for the experiments can be found in \cite{ebli2022}.

\begin{example}\label{ex:optimal-collapses}
In this example we consider the cell complex $\mc{X}$ in Figure~\ref{fig:optimal-up-collapses}.A-left, constructed as the alpha complex of points sampled uniformly at random in the cube $[0,1]\times [0,1]$. We work with the basis given the cells  of $\mc{X}$ and the standard boundary operator $\D$. The signal $s$ on the $1$-cells is given by the height function on the $1$-cells. The example illustrates a $(1,0)$-free sequential Morse matching $\M$ obtained by iterating Algorithm~\ref{alg:k-up} for $k=120$. Note that the optimal matchings correspond to $1$-cells where the signal is lower (see Figure~\ref{fig:optimal-up-collapses}.A-center). This can be explained by Remark~\ref{rmk:min-signal} and the fact that Equation~(\ref{eq:topolo_loss_compact}) favors collapsing cells with lower signal even when $\mc{X}$ is not a simplicial complex. 

The absolute value of the reconstruction error after the sequential Morse matching $\M$ is shown in Figure~\ref{fig:optimal-up-collapses}.B. As expected from Equation~(\ref{eq:topolo_loss_compact}), the error is mainly concentrated on the $1$-cells that are in the boundaries of the collapsed $2$-cells. Further, the map $\Phi^{\M}$ is an inclusion as showed in Lemma~\ref{lem:sparsification}. In panel C of Figure~\ref{fig:optimal-up-collapses} we  show the projection of the signal $s$ and the reconstructed signal $\Phi^{\M}\Psi^{\M}s$ on the Hodge decomposition. By Theorem~\ref{thm:reconstrcution} the signal is perfectly reconstructed on $\Ker\D_1=\Ker\Delta_1 \oplus \Ima \D_1^\dagger$, and only $\Ima\D_2$ contains non-trivial reconstruction error. Due to formatting constraints, we show the projection onto only $30$ (randomly chosen) vectors of the Hodge basis in $\Ima\D_1^\dagger$ and $\Ima\D_2$. 

\begin{figure}[!ht]
\vspace{2pt}
\begin{center}
\includegraphics[width=0.95\textwidth]{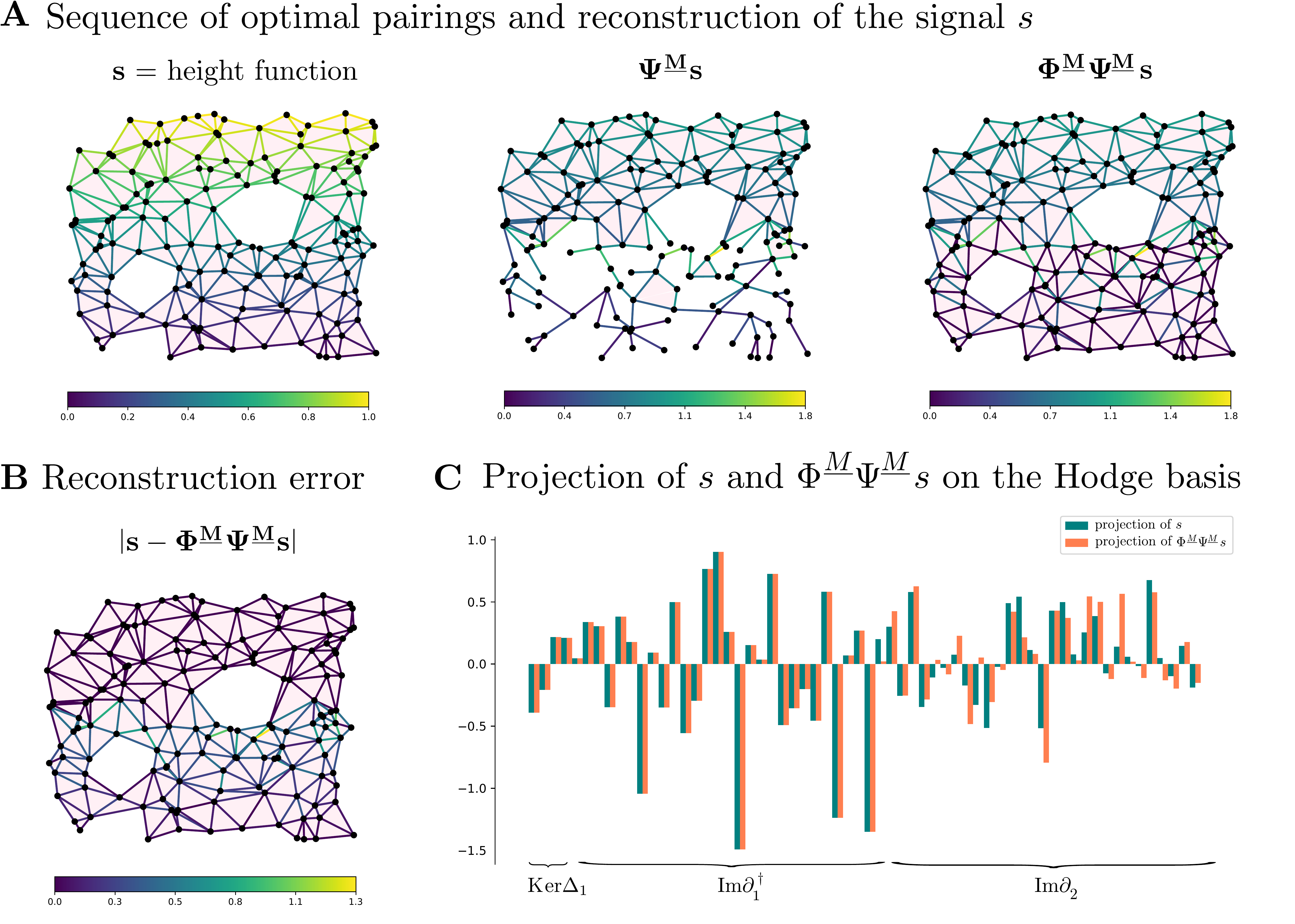}
    \caption{Optimal $(1,0)$-free sequential Morse matching $(M)$ obtained by iterating Algorithm~\ref{alg:k-up} for $k=120$ on $(2,1)$-pairs. The signal $s$ on the $1$-cells is given by the height function.  }
\label{fig:optimal-up-collapses}
\end{center}
\end{figure}

In Figure \ref{fig:optimal-up-collapses-uniform} we propose the same example as above with a non-geometric function on the $1$-cells. Specifically, the signal $s$ on the $1$-cells is given by sampling uniform at random in $[0,1]$ and the $(1,0)$-free sequential Morse matching $\M$ is obtained by iterating Algorithm~\ref{alg:k-up} until all $2$-cells were removed. 

\begin{figure}[!ht]
\vspace{2pt}
\begin{center}
\includegraphics[width=0.95\textwidth]{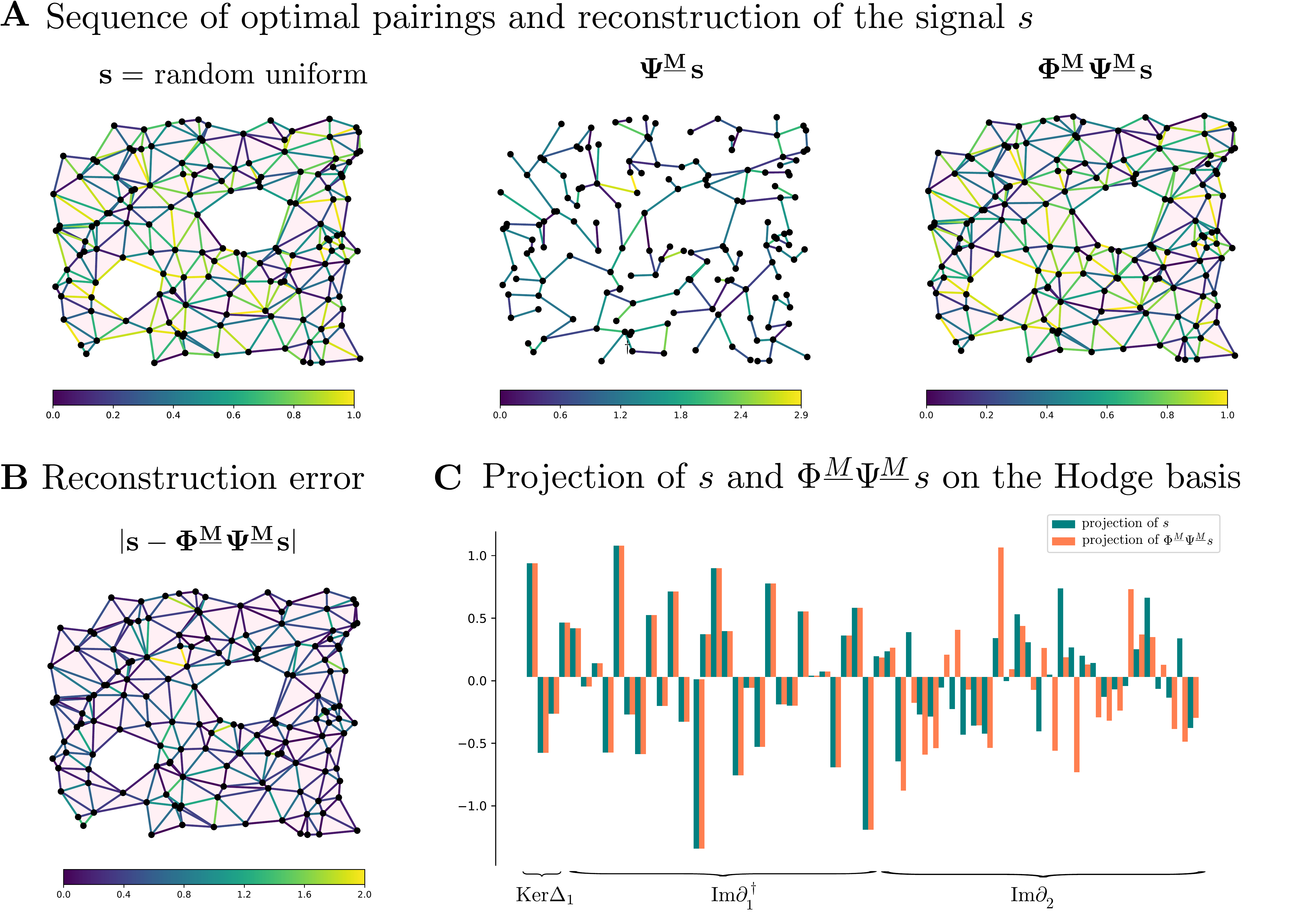}
    \caption{Optimal $(1,0)$-free sequential Morse matching $\M$ obtained by iterating Algorithm~\ref{alg:k-up} until all $2$-cells were removed. The signal $s$ on the $1$-cells is given by sampling uniform at random in $[0,1]$.  }
\label{fig:optimal-up-collapses-uniform}
\end{center}
\end{figure}

\end{example}

To quantify how low the topological reconstruction loss is after performing a sequential Morse matching with optimal pairings, we compare the reconstruction loss after a sequence of $k$ optimal matchings with the reconstruction loss after a sequence of $k$ random matchings. 

\begin{example} In this example we compare the sequence of optimal collapses presented in Example \ref{ex:optimal-collapses} in Figure \ref{fig:optimal-up-collapses} and in Figure \ref{fig:optimal-up-collapses-uniform} respectively with sequence of random collapses. In particular, we consider the complex $\mc{X}$ of Example~\ref{ex:optimal-collapses} with signal on the $1$-cells $s$ given by the height function as in Figure~\ref{fig:random-up-collapses} and signal $s$ given by sampling uniformly at random in $[0,1]$ as in Figure~\ref{fig:optimal-up-collapses-uniform}.
Instead of finding a sequence of $(2,1)$-pairings minimizing the reconstruction loss, at each step of algorithm \ref{alg:k-up} we will randomly remove a $(2,1)$-pair. We apply this procedure for $k=120$ iterations in case $s$ is the height function of the $1$-cells and until all $2$-cells are removed when the signal $s$ is sampled uniform at random in $[0,1]$.

Figure~\ref{fig:random-up-collapses}.A shows the projection on the Hodge basis of $s$ and $\Phi^{\M}  \Psi^{\M} s$ when $s$ is the height function and Figure~\ref{fig:random-up-collapses}.B shows the same result for $s$ sampled uniform at random. Due to formatting constraints, we show the projection onto only $30$ (randomly chosen) vectors of the Hodge basis in $\Ima\D_1^\dagger$ and $\Ima\D_2$. Note that, for both types of signal, the projection of the reconstructed signal $\Phi^{\M}\Psi^{\M} s$ and $s$ on $\Ima\D_2$ differ significantly more than the the projection on $\Ima\D_2$ of the reconstructed error and the signal in the case of optimal sequential Morse matching presented in Example~\ref{ex:optimal-collapses} (see Figure~\ref{fig:optimal-up-collapses}.D and Figure~\ref{fig:optimal-up-collapses-uniform}.D)
\begin{figure}[!ht]
\vspace{2pt}
\begin{center}
\includegraphics[width=0.95\textwidth]{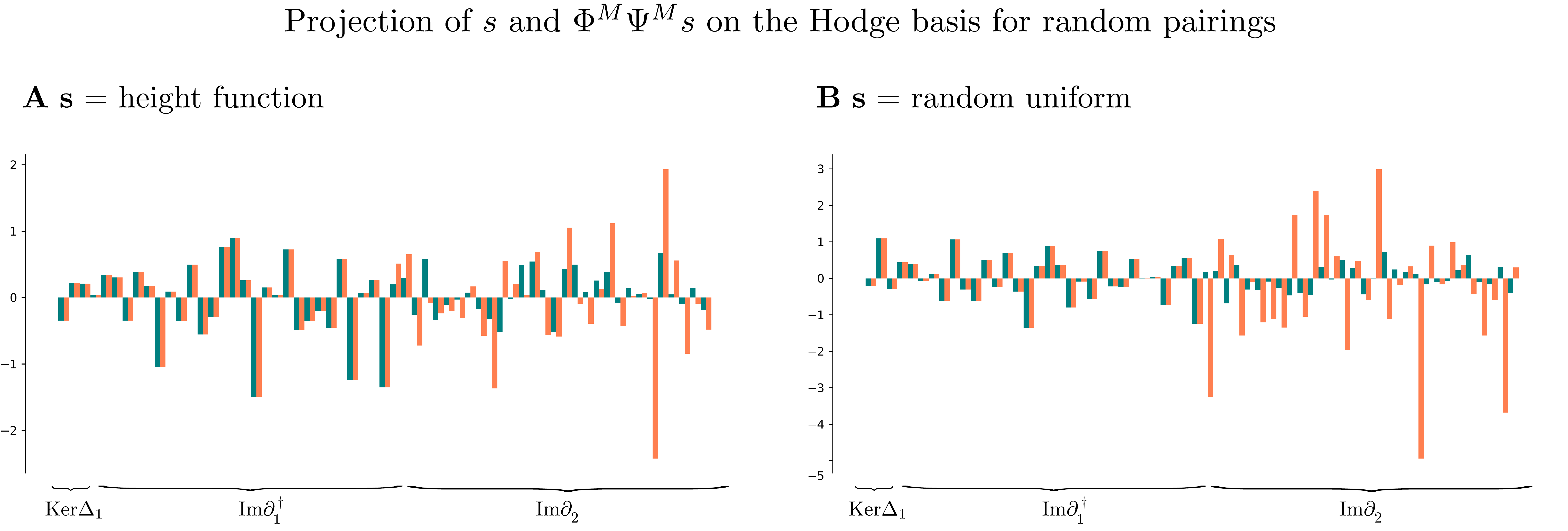}
\caption{Projection of the signal and the reconstructed signal on the Hodge basis after a sequence of random parings.}
\label{fig:random-up-collapses}
\end{center}
\vspace{-.25in}
\end{figure}
\end{example}

The quantitative results shown in the previous examples can be strengthened by comparing the value of the topological reconstruction loss for random and optimal sequence of pairings. In the next example we show that, for different types of both geometric and random signals, the topological reconstruction loss is significantly lower in sequentially optimal matchings than in random matchings.

\begin{example}
We consider again the same complex $\mc{X}$ as in Example~\ref{ex:optimal-collapses}.  Figure~\ref{fig:topo_error} shows the value of the topological reconstruction loss after a sequence optimal and random pairings. We took sequences of length $k=1,2,\dots 244$, terminating when all $2$-cells were reduced. In panel A we consider a signal on the $1$-cells sampled from a uniform distribution in $[0,1]$, in panel B the signal is the height function on the $1$-cells,  in panel C the signal is sampled from a normal distribution (mean $0.5$ and standard deviation 0.1), and in panel D the signal is given by the distance of the middle point of the $1$-cells from the center of the cube $[0,1]\times[0,1]$. The blue curve is the average over $10$ instantiations of optimal pairings while the green curve is the average over $10$ instantiations of random pairings. The filled opaque bars show the respective mean square errors. Note that for all type of functions, the loss for the optimal pairings is significantly lower than the loss of random pairings. 

\begin{figure}[!ht]
\vspace{4pt}
\centering

\begin{center}
\vspace{-15pt}
\includegraphics[width=0.95\textwidth]{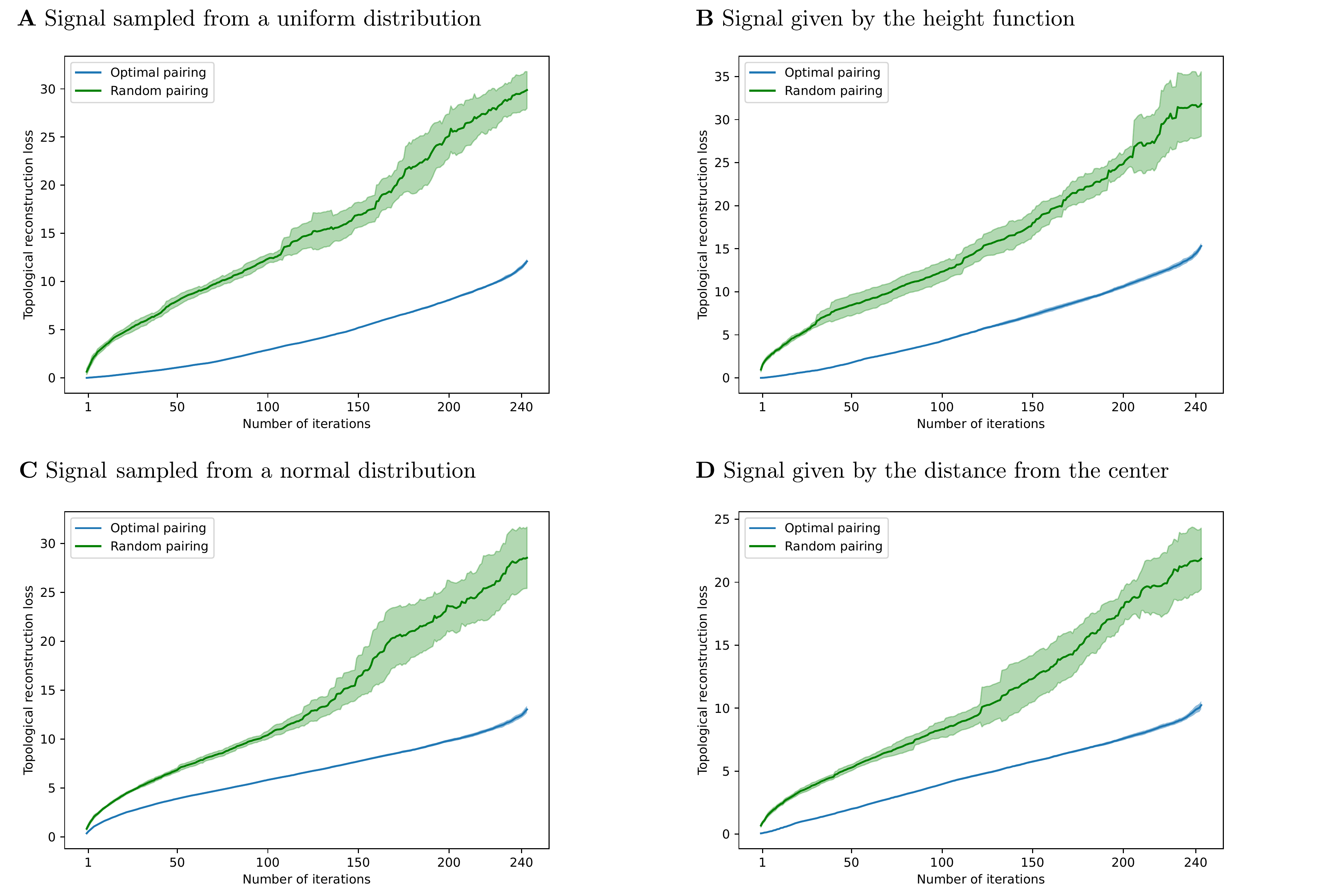}
\caption{Topological reconstruction error for sequences of optimal and random up-collapses with different lengths.} 
\label{fig:topo_error}
\end{center}
\vspace{-.25in}
\end{figure}
\end{example}

\section{Discussion} 

\paragraph{Contributions.}

The contributions of this paper are threefold. First we demonstrated that any deformation retract $(\Phi, \Psi)$ of finite-type based chain complexes over $\R$ is equivalent to a deformation retract $(\Phi^\mathcal{M}, \Psi^\mathcal{M})$ associated to a Morse matching $\mathcal{M}$ in a given basis. Second, we proved that the reconstruction error $s-\Phi\Psi s$, associated to any signal $s\in\C_n$ and deformation retract $(\Phi^\mc{M},\Psi^{M})$, is contained in specific components of the Hodge decomposition if and only if $\mc{M}$ is a $(n,n-1)$-free (sequential) Morse matching. In the more general case, we showed that the reconstruction error associated to a deformation retract of a based chain complex is contained in specific parts of the Hodge decomposition if and only if its Morsification $\mathcal{M}$ is $(n,n-1)$-free. Moreover, we proved that the composition $\Phi^M\Psi^M s$ can be thought as a sparsification of the signal $s$ in the $(n,n-1)$-free case. Finally, on the computational side, we designed and implemented algorithms that calculate (sequential) matchings that minimize the norm of topological reconstruction error. Further, we demonstrated computationally that finding a sequence of optimal matchings  with our algorithm performs significantly better than randomly collapsing.

\paragraph{Limitations.} The type of collapses that preserve cocycles involve chain maps, and those that preserve cycles involve the adjoints of these maps. This has two main limitations. The first one is that one can pick only one of the two features to be encoded at a time. The second limitation is the fact that chain maps do not necessarily send cocycles in $\C$ to cocycles in $\mathbf{D}$, and dually for cochain maps.

The proof of Theorem \ref{thm:reconstrcution} hints at the difficulties of trying to define chain maps that preserve cocycles and dually cochain maps that preserve cycles. Namely, to preserve cocycles with chain maps in dimension $n$, Morsification and Corollary \ref{reconstruction_closedform} yield some insight, saying that this will occur only when the paired $n$-cells of Morsification lie in $\D_n^\dag$. A sufficient condition for this is that $\Ker  \Psi \perp \Ima \Phi$, in which case $\restr{\D_n^\dagger}{\Ker \Psi} = (\restr{\D_n}{\Ker \Phi\Psi})^\dag$ (See Appendix \ref{adjoint_retraction_appendix}). This rarely occurs in the standard CW or sheaf bases.

\subsection{Applications and Future Work}

\paragraph{Algorithms for optimal collapses.} In this paper we minimize the reconstruction error by considering only single collapses. It would be desirable to find algorithms either for the optimal $(n,n-1)$-free Morse matchings, with no restriction on the length of the sequence, or for optimal $(n,n-1)$-free Morse matchings of given length $k$. We speculate that this task is likely to be NP-hard, given that the simpler task of finding a matching that minimises the number of critical cells is already known to be NP-hard \cite{Joswig2006ComputingOM,martinez2021survey}. In this case, it would be useful to develop algorithms to approximate optimal matchings. These could be then used to compare how far away the reconstruction error of a sequence of $k$ optimal pairings (Algorithm $\ref{alg:k-up}$) is from the reconstruction error of a optimal collapse of size $k$.

\paragraph{Applications with inner products.} In this paper, we have chosen examples that are helpful to visually illustrate the key results. However, the theory is built to accommodate a far larger class of applications. Examples where our theory may be useful for performing reductions that respect the inner product structure include the following
\begin{itemize}
    \item \textbf{Markov-based heat diffusion.} The foundational work of \cite{Coifman2006} introduces a graph-theoretic model of heat diffusion on a point cloud, and can be framed in terms of combinatorial (graph) Laplacians. Here, distance kernel functions induce a weighting function on the nodes and edges of fully connected graph over the points. This weighting function is equivalent to specifying an inner product on $\C$ where the standard basis vectors are orthogonal \cite{Horak2013}.
    \item \textbf{Triangulated manifolds.} If $M$ is a Riemannian manifold with smooth triangulation $K$, then $\C(K; \R)$ has an inner product structure that converges to the canonical inner product on the de Rham complex $\Omega(M)$ under a certain type of subdivision \cite{dodzuik1976}. This inner product on $\C(K; \R)$ -- and variations thereof -- are useful in discrete Exterior calculus and its applications \cite{hiptmair2002finite,hirani2003discrete}.
\end{itemize}
The main theorems of this paper will hold in any of the circumstances described above, and provide a discrete Morse theoretic procedure for signal compression that is aware of the geometric information contained in the inner product structure.

\paragraph{Pooling in cell neural networks.} Complementary to theoretical ideas, this research direction may have potential applications in pooling layers in neural networks for data structured on complexes or sheaves, such as in \cite{ebli2020simplicial, Bodnar2021WeisfeilerAL,han19sheaf}. One could use Algorithm \ref{alg:k-up} to reduce the complex for a fixed sized $k$ and then the map $\Phi$ to send the signal onto the reduced complex. We also envision that in pooling layers one could learn the $(n,n-1)$-free Morse matchings.

 \paragraph{Acknowledgements} K.\ M.\ has received funding from the European Union’s Horizon 2020 research and innovation program under the Marie Skłodowska-Curie grant agreement No 859860,\ S.\ E.\ and C.\ H.\ were supported by the NCCR Synapsy grant of the Swiss National Science Foundation. S.\ E.\ was also supported by Swiss National Science Foundation under grant No.\ 200021-172636 (Ebli)

The authors would like to acknowledge Kathryn Hess for her detailed feedback and insightful discussions.

\clearpage
\bibliography{biblio}
\bibliographystyle{plain} 
\clearpage

\appendix 
\section{Adjoints and Discrete Morse Theory} \label{inner_prod_app}

\subsection{Matrix Representation of Adjoints and Weights}\label{appendix:weights}

In this appendix we include a lengthier discussion about inner products and weight functions. To begin, we state a basic result about the matrix representation of the adjoint in finite-dimensional inner product spaces.

\begin{proposition}
    Let $V$ and $W$ be finite-dimensional inner product spaces where 
    $$ \langle v_1, v_2 \rangle_V = v_1^T A v_2$$
    and $$ \langle w_1, w_2 \rangle_W = w_1^T B w_2$$
    for some fixed bases of $V$ and $W$, where $A,B$ are positive definite symmetric matrices. If $T: V \to W$, then the adjoint $T^\dagger : W \to V$ of $T$ satisfies
    $$ T^\dagger = (A^{-1})^T T^T B^T.$$
\end{proposition}

The idea is that inner products are a vehicle to incorporate data with weights on the simplices into the linear algebraic world of combinatorial Laplacians. In particular, as mentioned in Remark \ref{rmk:weights}, there is a one-to-one correspondence between inner products where elementary simplicial (co)chains form an orthogonal basis and weight matrices on the simplices. In the literature there are two approaches to associate weights to the simplices. 

Firstly, the work of \cite{memoli} begins by letting $\D_n : \C_n(\mc{X}) \to \C_{n-1}(\mc{X})$ be the standard cellular boundary operator on a simplicial complex $\mc{X}$, and defines an inner product structure with respect to a basis given by the simplices via
$$\langle \sigma, \tau \rangle_n = \sigma^t W_n \tau,$$ where where each $W_n$ is a diagonal matrix. The diagonal entries of $W_n$ can be thought as weights on the $n$-cells. Then the coboundary operator $\D_n^\dagger : \C_{n-1}(\mc{X}) \to \C_n(\mc{X})$, is given by
$$ \D_n^\dagger = W_n^{-1} \D_n^T W_{n-1}$$
following the proposition above.

The second approach, exemplified by the work of \cite{Horak2013}, starts instead with the standard coboundary operator on a simplicial complex $\mc{X}$, $\delta_n = \D_n^T : \C_{n-1}(\mc{X}) \to \C_n(\mc{X})$. Here the inner product structure on $\C_n(\mc{X})$ with respect to a basis given by the simplices is defined instead to be
$$ \langle \sigma, \tau \rangle_n = \sigma^t W_n \tau,$$ where where each $W_n$ is a diagonal matrix,  the entries of which can be thought as weights on the $n$-cells. In this approach, the boundary operator is then written as
\begin{equation} \label{weighted_boundary_operator}
    \delta_n^\dagger = W_{n-1}^{-1} \delta_n^T W_n.
\end{equation}

Because we are working with discrete Morse theory, which conventionally is built for homology, we take the approach of always beginning with a boundary operator before constructing its adjoint operator. If one starts by defining a weighted boundary operator
$$ \tilde{\D}_n = W_{n-1}^{-1} \D_n W_{n},$$
then the adjoint operator induced by the weighted inner product yields
$$ \tilde{\D}_n^\dagger = W_{n}^{-1} W_{n} \D_n^T W_{n-1}^{-1} W_{n-1} = \D_n^T.$$
In other words, the adjoint of this weighted boundary operator is the standard coboundary operator, recovering the method of \cite{Horak2013}.

\subsection{The Adjoint of a Morse Retraction} \label{adjoint_retraction_appendix}

In this section, we explain why the orthogonality condition on the base $I$ of a based chain complex $\C$ is important for establishing a discrete Morse theoretic interpretation when taking adjoints in Theorem \ref{main_lemma}. One can of course take the adjoint of the maps in this theorem to construct a deformation retract of the adjoint cochain complex, along with a coboundary operator, cochain weak-equivalences, and a cochain homotopy between them. However, only in the special case of an orthogonal base can these maps be decomposed in terms of adjoint flow backwards along paths in the original matching graph $\mc{G}(\C)^M$.

\paragraph{Adjoint paths and flow.} Suppose we have a Morse matching $M$ on any based finite-type chain complex $\C$ over $\R$ with inner products. One can always define a notion of adjoint flow. First, observe that 
$$ \D_{\beta,\alpha} = 0 \Leftrightarrow \D_{\beta,\alpha}^\dagger = 0$$
and further $$\D^\dag_{\beta,\alpha} \text{ isomorphism } \Leftrightarrow \D_{\beta,\alpha} \text{ isomorphism}.$$ The opposite digraph $\mc{G}^{op}(\C)^{M}$ (same vertices with edges reversed) of the directed graph $\mc{G}(\C)^{M}$ then has an analogous relationship with the adjoint of the boundary operator. Namely, there is an edge $\beta \to \alpha$ whenever $\D_{\beta,\alpha}^\dag$ is non-zero, and a reversed edge $\beta \to \alpha$ in $\mc{G}^{op}(\C)^M$ whenever $\alpha \to \beta$ is in $M$ and $\D_{\beta,\alpha}^\dag$ is an isomorphism. The same cells are unpaired in the adjoint world as the original one, and thus the critical cells of both are the same.

For a directed path $\gamma = \alpha, \sigma_1, \ldots, \sigma_k, \beta$ in the graph $\mc{G}(\C)^M$, the \textit{adjoint index} $I^\dag(\gamma)$ of $\gamma$ is written as
$$ \mathcal{I}^\dag(\gamma) = \epsilon_0 \D_{\alpha, \sigma_0}^{\epsilon_0\dagger} \circ \ldots \circ \epsilon_1 \D_{\sigma_{n-2}, \sigma_{n-1}}^{\epsilon_{n-1}\dagger} \circ \epsilon_n \D_{\sigma_n, \beta}^{\epsilon_n\dagger} : C_\beta \to C_\alpha$$
where $k_i = -1$ if $\sigma_i \to \sigma_{i+1}$ is an element of $M$, and $1$ otherwise. For any $\alpha,\beta \in I$, we can interpret this as following the path backwards and taking the adjoint of each map. The adjoint of the summed index also has a similar structure:
$$\Gamma^\dag_{\beta,\alpha} = \sum_{\gamma : \alpha \to \beta} \mathcal{I}^\dag(\gamma) : C_\beta \to C_\alpha.$$ where the sum runs over all paths $\gamma$ from $\alpha \to \beta$ in $\mc{G}(\C)^M$ or, equivalently, over all paths $\beta \to \alpha$ in $\mc{G}^{op}(\C)^M$.

\paragraph{Main theorem for adjoint matching.} To see what can go wrong, we need to be careful to distinguish categorical projections -- those that simply delete components of a direct sum -- from orthogonal projections that arise from the inner product structure.

Let $f : C = \oplus_\alpha C_\alpha \to D = \oplus_\beta D_\beta$ be a map of finite-type graded Hilbert spaces, based by $I$ and $J$ respectively. Each component $f_{\beta,\alpha}$ can be thought of as the composition of maps
\begin{equation} \label{component_map} 
f_{\beta,\alpha} : C_\alpha \xrightarrow{i_\alpha} C \xrightarrow{f} D \xrightarrow{\pi_\beta} D_\beta
\end{equation}
such that we recover the total map $f$ via sums
$$ f = \sum_{\alpha,\beta} f_{\beta,\alpha}.$$
In a Hilbert space, the the inclusion $i_\alpha$ is adjoint to the \textit{orthogonal} projection $\mathrm{Proj}_{C_\alpha}$ onto $C_\alpha$ (Lemma \ref{adjoint_orthogonal_projection}), which not necessarily the categorical projection $\pi_\alpha$. The categorical projection map $\pi_\alpha$ agrees with $\mathrm{Proj}_{C_\alpha}$ if and only if 
\begin{equation} \label{orthogonality_condition}
    C_\alpha \perp C_{\alpha'}
\end{equation}
for all $\alpha' \in I \setminus \alpha$. If this equation holds for both $\alpha \in I$ and $\beta \in J$, then the adjoint of the component map
\begin{equation*}
(f_{\beta,\alpha})^\dagger : D_\beta \xrightarrow{\pi^\dag_\beta} D \xrightarrow{f^\dagger} D \xrightarrow{i^\dag_\alpha} D_\alpha.
\end{equation*}
agrees with the component maps of the adjoint
\begin{equation*}
(f^\dag)_{\alpha,\beta} : D_\beta \xrightarrow{i_\beta} D \xrightarrow{f^\dagger} D \xrightarrow{\pi_\alpha} D_\alpha.
\end{equation*}
If Equation $\ref{orthogonality_condition}$ holds for all $\alpha \in I$ and $\beta \in J$, then $$f^\dagger = \bigoplus_{\alpha, \beta} (f_{\beta,\alpha})^\dag.$$ In other words, the adjoint commutes with the direct sum.  

The reasoning above underpins why orthogonal components lead to a natural interpretation of the adjoint maps of \ref{main_lemma} in terms of the adjoint flow. If this is the case, we can take the adjoint of \ref{main_lemma} everywhere to prove the following important result.

\begin{theorem}[Sköldberg, \cite{Sko18}] \label{adjoint_main_lemma}
    Let $\C$ be a finite-dimensional chain complex indexed by an orthogonal base $I$, $M$ a Morse matching, and $$\C^M_n = \bigoplus_{\alpha \in I_n \cap M^0} C_\alpha.$$ The diagram
    
    \begin{center}

    \begin{tikzcd}
        
        \C^M \ar[r, shift right, "\Psi^\dagger"'] & \C \ar[l, shift right, "\Phi^\dagger"'] \ar[l, loop right, "h^\dagger"]
        
    \end{tikzcd}
     
    \end{center}
    
    is a deformation retract of cochain complexes, where for $x \in C_\beta$ with $\beta \in I_n$,
    \begin{itemize}
        \item $(\D^\dag_{\C^M})_n(x) = \sum_{\alpha \in M^0 \cap I_{n}} \Gamma^\dag_{\beta,\alpha}(x)$
        \item $\Phi_n^\dag(x) = \sum_{\alpha \in I_n} \Gamma^\dag_{\beta,\alpha}(x)$
        \item $\Psi_n^\dag(x) = \sum_{\alpha \in M^0 \cap I_n} \Gamma_{\beta, \alpha}^\dag(x)$
        \item $h_n^\dag(x) = \sum_{\alpha \in I_{n-1}}\Gamma^\dag_{\beta,\alpha}(x)$
    \end{itemize}
   
\end{theorem}

In most circumstances -- weighted Laplacians, cellular sheaves, etc. -- there is indeed an orthogonal basis. However, in the Morsification Lemma \ref{morsification}, we perform a reduction on the left component of
$$\Ker \Psi \oplus \Ima \Phi$$
which, in general, is \textit{not} orthogonal to $\Ima \Phi$. One needs to be careful in such situations not to utilise the adjoint flow decompositions given in Theorem \ref{adjoint_main_lemma}.

\end{document}